\documentclass[a4paper,11pt]{article}

\PassOptionsToPackage{utf8x}{inputenc}

\usepackage[a4paper,includeheadfoot,marginparwidth=20mm,marginparsep=2mm,left=20mm,right=30mm,top=30mm,bottom=30mm,headheight=16pt]{geometry}
\usepackage[T1]{fontenc}
\usepackage{amsmath,amssymb,amsthm}
\usepackage{fixmath}
\usepackage{bm}
\usepackage{graphicx}
\usepackage{booktabs}
\usepackage{ifthen}
\usepackage[textsize=tiny]{todonotes}
\usepackage{listings}
\usepackage{lmodern}

\let\counterwithin\relax
\usepackage{chngcntr}
\usepackage{placeins}
\usepackage{multirow}
\usepackage[nocompress]{cite}
\usepackage{titlesec}
\titleformat{\section}{\normalfont\sffamily\Large\bfseries}{\thesection}{1em}{}
\titleformat{\subsection}{\normalfont\sffamily\large\bfseries}{\thesubsection}{1em}{}
\titleformat{\subsubsection}{\normalfont\sffamily\normalsize\bfseries}{\thesubsubsection}{1em}{}
\titleformat{\paragraph}[block]{\normalfont\sffamily\normalsize\bfseries}{\theparagraph}{1em}{}
\usepackage[titles]{tocloft}

\renewcommand{\lstlistlistingname}{List of listings}
\makeatletter
\begingroup\let\newcounter\@gobble\let\setcounter\@gobbletwo
  \globaldefs\@ne \let\c@loldepth\@ne
  \newlistof{listings}{lol}{\lstlistlistingname}
\endgroup
\let\l@lstlisting\l@listings
\AtBeginDocument{\addtocontents{lol}{\protect\addvspace{10\p@}}}
\makeatother

\usepackage{fancyhdr}
\usepackage{subfig}

\graphicspath{{fig/}}

\makeatletter
\renewcommand\section{\suppressfloats\@startsection {section}{1}{\z@}%
  {-3.5ex \@plus -1ex \@minus -.2ex}%
  {2.3ex \@plus.2ex}%
  {\normalfont\sffamily\Large\bfseries}}
\renewcommand\subsection{\@startsection{subsection}{2}{\z@}%
  {-3.25ex\@plus -1ex \@minus -.2ex}%
  {1.5ex \@plus .2ex}%
  {\normalfont\sffamily\large\bfseries}}
\renewcommand\subsubsection{\@startsection{subsubsection}{3}{\z@}%
  {-3.25ex\@plus -1ex \@minus -.2ex}%
  {1.5ex \@plus .2ex}%
  {\normalfont\normalsize\sffamily\bfseries}}
\makeatother

\definecolor{codegreen}{rgb}{0,0.6,0}
\definecolor{codegray}{rgb}{0.5,0.5,0.5}
\definecolor{codepurple}{rgb}{0.58,0,0.82}
\definecolor{backcolour}{rgb}{0.95,0.95,0.92}
\definecolor{codeblue}{RGB}{0,160,238}
\definecolor{codeorange}{RGB}{217,83,25}
\lstdefinestyle{mystyle}{
    language=C++,
    commentstyle=\color{codeblue},
    keywordstyle=\color{codeorange},
    numberstyle=\tiny\color{codegray},
    breakatwhitespace=false,         
    breaklines=true,                 
    captionpos=b,                    
    keepspaces=true,                 
    numbers=left,                    
    numbersep=5pt,                  
    showspaces=false,                
    showstringspaces=false,
    showtabs=false,                  
    tabsize=2,
    basicstyle={\footnotesize\ttfamily},
    xleftmargin=.1\textwidth, 
    xrightmargin=.1\textwidth,
    morekeywords={omp,offload,target,parallel,simd,critical,atomic,linear,declare,uniform,reduction,offload_transfer},
    framextopmargin=2pt,
    frame=tb,
    escapeinside={(|}{|)},
    moredelim=[is][\color{codeorange}]{|*}{*|}
}
\newcommand{\vect}[1]{\ensuremath{{\bm{#1}}}}

\newcommand{\vf}{{\vect{f}}}
\newcommand{\vg}{{\vect{g}}}
\newcommand{\vh}{{\vect{h}}}

\newcommand{\vn}{{\vect{n}}}

\newcommand{\vr}{{\vect{r}}}
\newcommand{\vs}{{\vect{s}}}

\newcommand{\vv}{{\vect{v}}}
\newcommand{\vw}{{\vect{w}}}
\newcommand{\vx}{{\vect{x}}}
\newcommand{\vy}{{\vect{y}}}
\newcommand{\vz}{{\vect{z}}}

\newcommand{\valpha}{{\vect{\alpha}}}

\newcommand{\vpsi}{{\vect{\psi}}}

\newcommand{\vxi}{{\vect{\xi}}}

\newcommand{\vzero}{{\vect{0}}}

\newcommand{\dif}{\ensuremath{{\mathrm{d}}}}

\DeclareMathOperator{\diver}{div}

\DeclareMathOperator{\dist}{dist}

\DeclareMathOperator{\supp}{supp}
\DeclareMathOperator{\sign}{sign}
\DeclareMathOperator{\curl}{\mathbf{curl}}

\DeclareMathOperator{\erf}{erf}

\newcommand{\dtrace}{\ensuremath{\gamma_{0,\Sigma}^\mathrm{int}}}
\newcommand{\dtraceext}{\ensuremath{\gamma_{0,\Sigma}^\mathrm{ext}}}
\newcommand{\dtracem}{\ensuremath{\gamma_{0,\Sigma_m}^\mathrm{int}}}
\newcommand{\rinvdtrace}{\ensuremath{E_\Sigma}}

\newcommand{\rinvdtraceend}{\ensuremath{E_{\Sigma;,0}}}

\newcommand{\ntrace}{\ensuremath{\gamma_{1,\Sigma}^\mathrm{int}}}
\newcommand{\ntraceext}{\ensuremath{\gamma_{1,\Sigma}^\mathrm{ext}}}
\newcommand{\ntracem}{\ensuremath{\gamma_{1,\Sigma_m}^\mathrm{int}}}
\newcommand{\vdtrace}{\ensuremath{{\vect{\gamma}}_{0,\Sigma}^\mathrm{int}}}
\newcommand{\scurlm}{\ensuremath{\curl_{\Sigma_m}}}
\newcommand{\scurl}{\ensuremath{\curl_{\Sigma}}}
\newcommand{\nvecm}{\ensuremath{\vn_{m}}}

\makeatletter
\renewcommand*\env@matrix[1][\arraystretch]{%
  \edef\arraystretch{#1}%
  \hskip -\arraycolsep
  \let\@ifnextchar\new@ifnextchar
  \array{*\c@MaxMatrixCols c}}
\makeatother

\theoremstyle{plain}
\newtheorem{theorem}{Theorem}[section]
\newtheorem{corollary}[theorem]{Corollary}
\newtheorem{lemma}[theorem]{Lemma}
\newtheorem{proposition}[theorem]{Proposition}
\theoremstyle{definition}
\newtheorem{definition}[theorem]{Definition}

\theoremstyle{remark}
\newtheorem{remark}[theorem]{Remark}

\numberwithin{figure}{section}
\numberwithin{table}{section}
\AtBeginDocument{\counterwithin{lstlisting}{section}}

\usepackage{bbm}
\usepackage{doi}
\usepackage{enumerate}
\usepackage{hyperref}
\usepackage{url}

\usepackage{accents} 
\newif\ifshort \shorttrue 

\title{\sffamily\bfseries An integration by parts formula for the bilinear form of the hypersingular boundary integral operator for the transient heat equation in three spatial dimensions}
\author{Raphael Watschinger, G\"unther Of \\ Institute of Applied Mathematics, Graz University of Technology. \\ Steyrergasse 30, A-8010 Graz, Austria \\ \texttt{watschinger@math.tugraz.at}, \texttt{of@tugraz.at} }
\date{April 30, 2021}

\begin{document}

\pagenumbering{roman} 

\maketitle

\pagestyle{plain}
\setcounter{page}{1}
\pagenumbering{arabic}

\begin{abstract}
  While an integration by parts formula for the bilinear form of the hypersingular boundary integral operator for the transient heat equation in three spatial dimensions is available in the literature, a proof of this formula seems to be missing. Moreover, the available formula contains an integral term including the time derivative of the fundamental solution of the heat equation, whose interpretation is difficult at second glance. To fill these gaps we provide a rigorous proof of a general version of the integration by parts formula and an alternative representation of the mentioned integral term, which is valid for a certain class of functions including the typical tensor-product discretization spaces.
\end{abstract}



\section{Introduction}
The transient heat equation is the archetype of a parabolic partial differential equation in space and time. Nevertheless, it has quite a few things in common with elliptic partial differential equations in space like the Laplace equation, when it comes to integral equations. For example, one can define the single layer boundary integral operator $V$ and hypersingular boundary integral operator $D$ for the heat equation as integral operators on a lateral space-time boundary ${\Sigma = \partial \Omega \times (0,T) \subset \mathbb{R}^d \times \mathbb{R}}$ in a similar way as for elliptic partial differential equations. Surprisingly, one can even show that both operators are elliptic in suitable anisotropic Sobolev spaces~\cite{ArnNoo1989,Cos1990}. This makes it particularly interesting to consider Galerkin variational formulations for the solution of integral equations of the form $Vq = g$ or $D u = h$ and related boundary element methods.

Evaluating the hypersingular operator $D$ for the heat equation is as challenging as in the elliptic case. The problem is that the operator cannot be expressed as an integral in a classical sense. In case of the Laplace equation and other elliptic partial differential equations there exist several approaches to regularize the hypersingular operator. One particularly interesting strategy is to regularize the bilinear form associated with $D$ instead of the operator itself. Some sort of integration by parts is applied and eventually leads to a representation in terms of weakly singular integrals in a discrete setting. This approach has been applied to a wide class of partial differential equations, e.g., for the Laplace equation~\cite{Nedelec}, the Helmholtz equation~\cite{Maue, NedelecBook}, time-harmonic Maxwell equations~\cite{NedelecBook} and linear elastostatics~\cite{Han, Kupradze}.

For the hypersingular operator of the heat equation or rather the corresponding bilinear form  an integration by parts formula is also available. In \cite{Cos1990} a formula for the 2+1D case, i.e.~two space dimensions plus the additional time dimension, is provided together with an outline of its proof. A formula for the 3+1D case can be found for example in \cite{Doh2019} and \cite{Mes2014}, but to the best of our knowledge no proof is provided in the literature. In addition, the formula in the mentioned works contains a boundary integral including the time derivative of the fundamental solution of the heat equation, which per se is locally not integrable on the considered integration domain. This makes it difficult to understand the formula in a general setting. In this work we want to fill these gaps by giving a rigorous proof of the integration by parts formula in 3+1D in a rather general form. The problematic integral term with the time derivative will appear here in a general form, whose evaluation is again difficult for non-smooth functions. However, we will derive an integral representation which is valid for a certain class of functions including the ones typically used for discretization and overcomes the problem of the locally non integrable time derivative.

The remaining paper is structured as follows. Section~\ref{sec:preliminaries} serves as preparation for the rest of the paper. Here we introduce the relevant function spaces and provide a few results which will be used in the main proofs. In Section~\ref{sec:heat_equation} we focus on the transient heat equation. We discuss solvability aspects and introduce the boundary integral operators, including the hypersingular operator $D$. As a side result we give a proof of Theorem~\ref{thm:extended_maximum_principle}, which is a generalization of the classical parabolic maximum principle and is needed later on. In Section~\ref{sec:ibp_formula} we finally consider the general integration by parts formula for the bilinear form of $D$, which is formulated in Theorem~\ref{thm:general_ibp_formula} and proven in Section~\ref{subsec:proof_general_ibp_formula}. The aforementioned 'time derivative term' of this formula, will be further investigated in Section~\ref{sec:bf_time_derivative}. Here we will give the details why the formulation in~\cite{Doh2019,Mes2014} is not adequate in general and provide our alternative in Theorem~\ref{thm:second_bilinear_form_tensor_case}. Section~\ref{sec:conclusion} concludes the paper with a short summary and outlook.

\section{Preliminaries} \label{sec:preliminaries}
In this section we introduce the basic notation used throughout the paper and discuss a few concepts and results which we need for the main proofs of the paper.

\subsection{Anisotropic Sobolev spaces, trace operators and a surface curl} \label{sec:spaces_traces_curl}
For an open set $A \subset \mathbb{R}^d$ we denote by $C(A)$, $C^k(A)$ and $C^k(\overline{A})$ for $k \in \mathbb{N} \cup \{\infty\}$ the usual sets of continuous and $k$ times continuously differentiable functions on $A$ or $\overline{A}$ in the appropriate sense. Functions in $C^\infty(A)$ with compact support are denoted by $C^\infty_c(A)$. By $L^p(A)$, $p \in [1,\infty]$ we denote the standard Lebesgue spaces on $A$. When referring to functions in these spaces, we always mean the related equivalence classes. We use bold face letters to denote spaces containing functions mapping to $\mathbb{R}^n$, whose components are in the respective function space. Whenever we consider such functions, the dimension $n$ is clear from context, so we do not specify it.

Throughout this work, let $\Omega \subset \mathbb{R}^3$ be a bounded Lipschitz domain as in Definition~\ref{def:lipschitz_domain} with boundary~$\Gamma$, let $T > 0$ be finite and $Q:=\Omega \times (0,T)$ be the space-time cylinder with lateral boundary $\Sigma := \Gamma \times (0,T)$. In addition let $\alpha > 0$. All our considerations in this paper are related to the initial boundary value problem for the heat equation
\begin{alignat}{2}
  \frac{\partial}{\partial t}u - \alpha \upDelta u &= 0 \qquad && \text{in } Q, \label{eq:heat_equation} \\
  u(\cdot, 0)   &= 0 && \text{in } \Omega, \label{eq:initial_condition}
\end{alignat}
with an additional Dirichlet or Neumann boundary condition
\begin{subequations} \label{eq:boundary_condition}
\begin{alignat}{2}
  u &= g \qquad &&\text{on } \Sigma, \label{eq:dirichlet_boundary_values} \\ 
  \ntrace u &= h \qquad &&\text{on } \Sigma, \label{eq:neumann_boundary_values}
\end{alignat}
\end{subequations}
respectively. For the study of these problems we consider the anisotropic Sobolev space
\begin{equation*}
  H^{1,1/2}(Q) := L^2(0,T;H^1(\Omega)) \cap H^{1/2}(0,T;L^2(\Omega)),
\end{equation*}
and its subspaces 
\begin{align*}
  H^{1,1/2}_{;0,}(Q) &:= L^2(0,T;H^1(\Omega)) \cap H^{1/2}_{0,}(0,T;L^2(\Omega)), \\
  H^{1,1/2}_{;,0}(Q) &:= L^2(0,T;H^1(\Omega)) \cap H^{1/2}_{,0}(0,T;L^2(\Omega)),
\end{align*}
where we use the standard notation for Bochner spaces and Sobolev spaces. The space $H^{1,1/2}_{;0,}(Q)$ can be interpreted as the subspace of $H^{1,1/2}(Q)$ containing functions which are $0$ at~$t=0$, or, to be more precise, the functions whose extension by zero for $t<0$ is in $H^{1,1/2}(\Omega \times (-\infty,T))$, see e.g.~\cite[Proposition~5.2]{LioMag1972}. The space $H^{1,1/2}_{;,0}(Q)$ contains functions vanishing at $t=T$ in the same sense. Particular representations of the norms of these anisotropic spaces are not needed here, so we only refer to \cite{Cos1990} and \cite{DohNiiSte2019} for them. Note that we use the notation of~\cite{DohNiiSte2019}. In~\cite{Cos1990}, $\widetilde{H}^{1,1/2}(Q)$ and $\accentset{(T)}{H}^{1,1/2}(Q)$ denote the spaces equivalent to $H^{1,1/2}_{;0,}(Q)$ and $H^{1,1/2}_{;,0}(Q)$.

The following density result will be used several times throughout the paper.
\begin{proposition} \label{prop:density_anisotropic_Q}
  The space 
  \begin{equation}
    C^\infty_c(\overline{\Omega}\times(0,T)) 
      := \{ u: u=\widetilde{u}|_Q, \widetilde{u} \in C^\infty_c(\mathbb{R}^3\times(0,T))\}
  \end{equation}
   is dense in $H^{1,1/2}(Q)$, $H^{1,1/2}_{;0,}(Q)$, and $H^{1,1/2}_{;,0}(Q)$.
\end{proposition}
\begin{proof}[Sketch of the proof]
  In \cite[Proof of Lemma 2.22]{Cos1990} it is mentioned, that 
  \begin{equation*}
    C^\infty_c(\overline{\Omega}\times(0,T]) 
      := \{u: u=\widetilde{u}|_Q, \widetilde{u} \in C^\infty_c(\mathbb{R}^3\times(0,\infty)\}
  \end{equation*}
  is dense in $H^{1,1/2}_{;0,}(Q)$, and that this follows from the density of $C^\infty(\overline{\Omega})$ in~$H^1(\Omega)$ and the density of $C^\infty_c((0,T])=\{f : f = \widetilde{f}|_{(0,T)}, \widetilde{f} \in C^\infty_c(0,\infty)\}$ in $H^{1/2}_{0,}(0,T)$ by tensor product arguments. The same tensor product arguments can be used to show the three density results stated above, since $C^\infty_c(0,T)$ is dense in $H^{1/2}(0,T)$, see e.g.~\cite[Theorem~1.4.2.4]{Gri2011}, and also in $H^{1/2}_{0,}(0,T)$ and $H^{1/2}_{,0}(0,T)$, see \cite[Theorem~2.2.2]{Zan2019}.
\end{proof}
Following \cite{DohNiiSte2019}, we consider in addition the anisotropic Sobolev space 
\begin{equation*}
  H^{1/2,1/4}(\Sigma) := L^2(0,T;H^{1/2}(\Gamma)) \cap H^{1/4}(0,T;L^2(\Gamma))
\end{equation*}
on $\Sigma$ with the norm
\begin{equation*}
  \|\varphi\|_{H^{1/2,1/4}(\Sigma)} := \left( \int_0^T \|\varphi(\cdot,t)\|^2_{H^{1/2}(\Gamma)}\,\dif t 
    + \int_0^T \int_0^T \frac{\|\varphi(\cdot,t) - \varphi(\cdot,\tau)\|^2_{L^2(\Gamma)}}{|t-\tau|^{3/2}} \,\dif \tau\,\dif t
  \right)^{1/2},
\end{equation*}
where $\|\cdot\|_{H^{1/2}(\Gamma)}$ denotes the usual Sobolev--Slobodeckij norm on $\Gamma$, and its dual
\begin{equation*}
  H^{-1/2,-1/4}(\Sigma) = (H^{1/2,1/4}(\Sigma))'.
\end{equation*}
By $\langle \cdot, \cdot \rangle_\Sigma$ we denote the duality product on $H^{-1/2,-1/4}(\Sigma) \times H^{1/2,1/4}(\Sigma)$ which is understood as the continuous extension of the $L^2$ inner product 
\begin{equation*}
  \langle \psi, \varphi \rangle_{L^2(\Sigma)} = \int_0^T \int_\Gamma \psi(\vx,t) \varphi(\vx,t)\,\dif\vs_\vx\,\dif t  
\end{equation*}
from $L^2(\Sigma) \times H^{1/2,1/4}(\Sigma)$ to $H^{-1/2,-1/4}(\Sigma) \times H^{1/2,1/4}(\Sigma)$. 

Our interest in the space $H^{1/2,1/4}(\Sigma)$ is explained by the fact that it can be interpreted as the trace space of $H^{1,1/2}_{;0,}(Q)$. Indeed, the following theorem holds.
\begin{theorem}[\!\!{\cite[Theorem 2.1]{LioMag1972}, \cite[Lemma 2.4]{Cos1990}}] \label{thm:dtrace}
  There exists a unique continuous operator $\dtrace$ from $H^{1,1/2}_{;0,}(Q)$ to $H^{1/2,1/4}(\Sigma)$ such that $\dtrace u = u|_{\Sigma}$ for all $u \in C^\infty_c(\overline{\Omega}\times(0,T))$. This operator is surjective. 
\end{theorem}
\begin{remark}
  The trace operator $\dtrace$ can also be considered as a surjective operator from $H^{1,1/2}(Q)$ or~$H^{1,1/2}_{;,0}(Q)$ to $H^{1/2,1/4}(\Sigma)$. In a slight abuse of notation, we denote all three operators by $\dtrace$.
\end{remark}
\begin{corollary} \label{cor:right_inverse_dtrace}
  There exist continuous operators
  \begin{align*}
    \rinvdtrace&: H^{1/2,1/4}(\Sigma) \rightarrow H^{1,1/2}(Q), \\
    \rinvdtraceend&: H^{1/2,1/4}(\Sigma) \rightarrow H^{1,1/2}_{;,0}(Q),
  \end{align*}
  which are right-inverses of the operator $\dtrace$ on the respective space on $Q$.
\end{corollary}

For the definition of the Neumann trace operator we introduce the space
\begin{equation*}
  H^{1,1/2}_{;0,}(Q,\partial/\partial t - \alpha \upDelta) = \left\{ u \in H^{1,1/2}_{;0,}(Q) : \frac{\partial }{\partial t} u - \alpha \upDelta u \in L^2(Q) \right\}
\end{equation*}
with the usual norm, see \cite{Cos1990, DohNiiSte2019}, and consider Green's first identity for the heat equation, which reads
\begin{align*}
  \alpha \int_0^T \int_\Gamma (\vn \cdot \nabla u)\, v\,\dif\vs_\vx\,\dif t
  = &-\int_0^T \int_\Omega \bigg(\frac{\partial u}{\partial t}- \alpha \upDelta u\bigg) v\,\dif\vx\,\dif t
    + \alpha\, \int_0^T \int_\Omega \nabla u \cdot \nabla v \,\dif\vx\,\dif t \\
    &+\int_0^T \int_\Omega \frac{\partial u}{\partial t}v \,\dif\vx\,\dif t
\end{align*}
for functions $u \in C^2(\overline{Q})$ and $v \in C^1(\overline{Q})$. As usual, we can use this identity to generalize the Neumann trace $\vn \cdot \nabla u$ for $u \in H^{1,1/2}_{;0,}(Q,\partial/\partial t - \alpha \upDelta)$, if we can ensure that the right-hand side is well-defined and continuous for such $u$ and suitable $v$. This is not immediately clear for the last integral on the right-hand side, but is established in the following proposition.
\begin{proposition}[\!\!{\cite[cf.~Lemma 2.6]{Cos1990}}] \label{prop:bilinear_form_d}
  The bilinear form 
  \begin{equation} \label{eq:def_bilinear_form_d}
    d(u,v) := \int_0^T\int_\Omega \frac{\partial u}{\partial t}(\vx,t) v(\vx,t) \,\dif\vx\,\dif t
  \end{equation}
  can be continuously extended from $C^\infty_c(\overline{\Omega}\times(0,T]) \times C^\infty_c(\overline{\Omega}\times[0,T))$ to $H^{1,1/2}_{;0,}(Q) \times H^{1,1/2}_{;,0}(Q)$ and is bounded by a constant $c_d$, which does not depend on $\Omega$. 
\end{proposition}
A sketch of the proof of Proposition~\ref{prop:bilinear_form_d} is given in Section~\ref{sec:appendix_proofs}. Here we continue with the definition of the Neumann trace operator $\ntrace$.
\begin{proposition}[\!\!{\cite[cf.~Lemma 2.16]{Cos1990}}]
  The map 
  \begin{equation*} 
    \ntrace:\ H^{1,1/2}_{;0,}(Q,\partial/\partial t - \alpha \upDelta) \rightarrow H^{-1/2,-1/4}(\Sigma)  
  \end{equation*}
  defined by
  \begin{equation} \label{eq:def_ntrace}
      \langle \ntrace u, \psi \rangle_\Sigma 
        := -\int_0^T \int_\Omega \left( \rinvdtraceend \psi \bigg(\frac{\partial}{\partial t}- \alpha \upDelta \bigg)u
      -\nabla u \cdot \nabla (\rinvdtraceend \psi)\right) \,\dif\vx\,\dif t + d(u,\rinvdtraceend \psi)
  \end{equation}
  for all $u \in H^{1,1/2}_{;0,}(Q,\partial/\partial t - \alpha \upDelta)$ and $\psi \in H^{1/2,1/4}(\Sigma)$ is well-defined and continuous. In particular, it does not depend on the choice of the extension $\rinvdtraceend$ from $H^{1/2,1/4}(\Sigma)$ to $H^{1,1/2}_{;,0}(Q)$. Furthermore, for $u \in C^2(\overline{Q})$ there holds $\ntrace u = \vn \cdot \nabla u|_\Sigma$.
\end{proposition}
For later reference, we introduce the surface curl of a function in $H^{1/2,1/4}(\Sigma)$. We define it using a weak variational definition, inspired by the definitions of the purely spatial tangential trace and surface curl in \cite{SayBroHas2019}, see the definition of $\gamma_T$ and $\nabla_\Gamma^\perp$ in Sections~16.2 and 16.10, respectively.
\begin{definition}
  The \emph{surface curl} $\scurl \varphi \in \mathbf{H}^{-1/2,-1/4}(\Sigma)$ of a function $\varphi \in H^{1/2,1/4}(\Sigma)$ is defined by
  \begin{equation} \label{eq:def_scurl}
    \langle \scurl \varphi, \vpsi \rangle_\Sigma 
      = \langle \nabla \rinvdtrace \varphi, \curl( \rinvdtrace \vpsi) \rangle_{\mathbf{L}^2(Q)} \quad \text{for all } \vpsi \in \mathbf{H}^{1/2,1/4}(\Sigma), 
  \end{equation}
  where $\rinvdtrace$ is the continuous right inverse of $\dtrace$ from Corollary~\ref{cor:right_inverse_dtrace} and its application to a vector valued function is understood componentwise.
\end{definition}
\begin{proposition} \label{prop:surface_curl}
  The operator $\scurl: H^{1/2,1/4}(\Sigma) \rightarrow \mathbf{H}^{-1/2,-1/4}(\Sigma)$ is well-defined and continuous. In particular, \eqref{eq:def_scurl} is independent of the extension~$\rinvdtrace$. If $\widetilde{\varphi} \in C^2(\overline{Q})$ and $\varphi = \widetilde{\varphi}|_{\Sigma}$, then there holds
  \begin{equation} \label{eq:scurl_smooth_functions}
    \scurl \varphi = \nabla \widetilde{\varphi} \times \vn.
  \end{equation}
\end{proposition}
\noindent The proof of this proposition is again given in Section~\ref{sec:appendix_proofs} in the appendix.
\subsection{Selected results from distribution theory}
For the proof of the integration by parts formula in Theorem~\ref{thm:general_ibp_formula} we collect a few definitions and results of distribution theory, which can be found in a standard textbook like \cite{Tre1970} to which we refer for the missing details.

By $\mathcal{D}'(A)$ we denote the distributions on an open set $A \subset \mathbb{R}^d$, which are those linear functionals on $C^\infty_c(A)$ which are sequentially continuous with the usual notion of convergence in $C^\infty_c(A)$, see e.g.~\cite[page~65]{McL2000}. In the same manner we define $\mathcal{E}'(A)$ as the set of all linear, sequentially continuous functionals on $C^\infty(A)$. We use the notation $u[w]$ for the application of~$u$ in~$\mathcal{D}'(A)$ or $\mathcal{E}'(A)$ to a function~$w$ in $C^\infty_c(A)$ or $C^\infty(A)$, respectively. 

Let $L^1_{\mathrm{loc}}(A)$ be the set of all measurable functions $u$ on $A$ such that $\|u\|_{L^1(K)}< \infty$ for all compact subsets $K$ of $A$. For each $u \in L^1_{\mathrm{loc}}(A)$ we can define a distribution by setting
\begin{equation*}
  u[w] := \int_A u \, w\, \dif \vx.
\end{equation*}
A distribution that can be represented by a function in $L^1_{\mathrm{loc}}(A)$ in this way is called regular.

For a multi-index $\valpha \in \mathbb{N}_0^d$ the derivative $D^\valpha$ of a distribution $S \in \mathcal{D}'(A)$ is defined by
\begin{equation*}
  D^\valpha S [w] = (-1)^{|\valpha|}S [D^\valpha w]
\end{equation*}  
for all $w \in C^\infty_c(A)$, where $|\valpha|:=\sum_j |\alpha_j|$. The derivative $D^\valpha S$ is itself a distribution. 

The restriction $S|_B$ of a distribution $S \in \mathcal{D}'(A)$ to an open subset $B$ of $A$ is defined by setting $S|_B[w] = S[\widetilde{w}]$ for all $w \in C^\infty_c(B)$, where $\widetilde{w}$ denotes the extension by zero of $w$ to $A$. There holds $S|_B \in \mathcal{D}'(B)$. The support $\supp(S)$ of a distribution $S \in \mathcal{D}'(A)$ is defined as the largest relatively closed subset $F$ of $A$ such that $S|_{A\backslash F} = 0$. In \cite[Theorem~24.2]{Tre1970} it is shown that
\begin{equation*}
  \mathcal{E}'(A) = \{S \in \mathcal{D}'(A) : \supp(S) \text{ is a compact subset of } A\}.
\end{equation*}

In the following we focus on distributions on the whole space $\mathbb{R}^d$. We define the convolution of a distribution $S \in \mathcal{D}'(\mathbb{R}^d)$ and a test function $u \in C^\infty_c(\mathbb{R}^d)$ by
\begin{equation*}
  S \ast u :\ \vx \mapsto S_{\vy}[u(\vx - \cdot_\vy)],  
\end{equation*}
where the index $\vy$ added to $S$ indicates that it acts with respect to this variable. It can be shown that this is a function in $C^\infty(\mathbb{R}^d)$ or even in $C^\infty_c(\mathbb{R}^d)$ if $\supp(S)$ is compact, see e.g.~\cite[Theorem~27.3]{Tre1970}. For $S \in \mathcal{D}'(\mathbb{R}^d)$, $R \in \mathcal{E}'(\mathbb{R}^d)$ and $u \in C^\infty_c(\mathbb{R}^d)$ we can also consider the functions 
\begin{equation*}
  S_\vy[u(\cdot_\vx+\cdot_\vy)] :\ \vx \mapsto S_\vy[u(\vx + \cdot_\vy)]
\end{equation*}
and $R_\vy[u(\cdot_\vx+\cdot_\vy)]$, which are in $C^\infty(\mathbb{R}^d)$ and $C^\infty_c(\mathbb{R}^d)$, respectively. This follows directly from the alternative representation $S_\vy[u(\vx+\cdot_\vy)] = (S \ast \check{u})(-x)$, where $\check{u}(\vx) := u(-\vx)$. In particular, we can define the convolution of $S \in \mathcal{D}'(\mathbb{R}^d)$ and $R \in \mathcal{E}'(\mathbb{R}^d)$ as an element in $\mathcal{D}'(\mathbb{R}^d)$ via 
\begin{equation} \label{eq:def_convolution_distributions}
  (S \ast R)[u] := S_\vx[ R_\vy[ u(\cdot_\vx+ \cdot_\vy)]] \qquad \text{for all } u \in C^\infty_c(\mathbb{R}^d)
\end{equation}
and similarly $(R \ast S) \in \mathcal{D}'(\mathbb{R}^d)$ by
\begin{equation*}
  (R \ast S)[u] := R_\vx[ S_\vy[ u(\cdot_\vx + \cdot_\vy)]] \qquad \text{for all } u \in C^\infty_c(\mathbb{R}^d).
\end{equation*}
Let us collect some properties of the convolution of distributions defined in this way.

\begin{proposition}[\!\!{\cite[Theorem~27.4, Propositions~27.3 and 27.5]{Tre1970}}] \label{prop:properties_convolution_distribution}
  Let $S \in \mathcal{D}'(\mathbb{R}^d)$, $R \in \mathcal{E}'(\mathbb{R}^d)$ and $\valpha \in \mathbb{N}_0^d$. Then
  \begin{align*}
    S \ast R &= R \ast S, \\
    D^\valpha (S \ast R) &= (D^\valpha S) \ast R = S \ast (D^\valpha R), \\
    \delta_\vzero \ast S &= S,
  \end{align*}
  where $\delta_\vzero$ is the delta distribution defined by $\delta_\vzero[w] = w(0)$ for all $w \in C^\infty(\mathbb{R}^d)$.
\end{proposition}

Two distributions playing an important role in this paper are the adjoint operators~$(\dtrace)'$ and $(\ntrace)'$ of the trace operators. If we interpret the Dirichlet trace operator $\dtrace$ as a continuous operator from ${C^\infty(\mathbb{R}^3 \times \mathbb{R})}$ to $H^{1/2,1/4}(\Sigma)$, its adjoint $(\dtrace)': H^{-1/2,-1/4}(\Sigma) \rightarrow \mathcal{E}'(\mathbb{R}^3\times\mathbb{R})$ is given by 
\begin{equation} \label{eq:adjoint_dtrace}
  (\dtrace)'(\varphi)[w] := \langle \varphi, \dtrace w \rangle_\Sigma
\end{equation} 
for $\varphi \in H^{-1/2,-1/4}(\Sigma)$ and $w \in C^\infty(\mathbb{R}^3\times\mathbb{R})$. Similarly, we consider the Neumann trace operator~$\ntrace$ as a continuous operator from ${C^\infty(\mathbb{R}^3 \times \mathbb{R})}$ to $H^{-1/2,-1/4}(\Sigma)$ by setting ${\ntrace u = \vn \cdot \nabla u}$ for $u \in C^\infty(\mathbb{R}^3\times\mathbb{R})$, and end up with its adjoint $(\ntrace)':H^{1/2,1/4}(\Sigma) \rightarrow \mathcal{E}'(\mathbb{R}^3\times\mathbb{R})$ defined by
\begin{equation} \label{eq:adjoint_ntrace}
  (\ntrace)'(\psi)[w] := \langle \ntrace w, \psi \rangle_\Sigma
\end{equation}
for $\psi \in H^{1/2,1/4}(\Sigma)$ and $w \in C^\infty(\mathbb{R}^3\times\mathbb{R})$. 

\section{The transient heat equation in space-time -- selected results} \label{sec:heat_equation}
In Section \ref{sec:spaces_traces_curl} we have introduced the initial-boundary value problem \eqref{eq:heat_equation}--\eqref{eq:boundary_condition} for the transient heat equation. Here we want to give some more details about its solution with a focus on results which are relevant for our later considerations. The first theorem which we state is a classical existence and uniqueness result, which can be found in a more general form in \cite[Theorem~6.2.8]{Are2011}. The second result known as the parabolic maximum principle is also standard, see e.g.~\cite[Section~2.3.3, Theorem 4]{Eva1998}.
\begin{theorem} \label{thm:classical_solvability_heat_equation}
  Let $g \in C(\overline{\Sigma})$ such that $g(\vx,0) = 0$ for all $\vx \in \Gamma$. Then the Dirichlet initial boundary value problem \eqref{eq:heat_equation}--\eqref{eq:dirichlet_boundary_values} admits a unique classical solution $u \in C(\overline{Q}) \cap C^\infty(Q)$. 
\end{theorem}
\begin{proposition} \label{prop:maximum_principle}
  Let $u \in C(\overline{Q}) \cap C^2(Q)$ satisfy \eqref{eq:heat_equation}. Then
  \begin{align*}
    \max_{(\vx,t) \in \overline{Q}} u(\vx,t) = \max_{(\vx,t) \in (\overline{\Sigma} \cup (\overline{\Omega} \times \{0\}))} u(\vx,t) 
  \end{align*}
  and the same holds if the maximum is replaced by the minimum on both sides.
\end{proposition}

Switching to the setting of the anisotropic Sobolev spaces introduced in Section~\ref{sec:spaces_traces_curl} allows us to consider more general Cauchy data and retain the unique solvability. A proof of the following result can be found for example in \cite{Cos1990}, see Theorem~2.9 for the Dirichlet problem and Corollary~3.17 for the Neumann problem.
\begin{theorem} \label{thm:unique_solvability_heat_equation}
  Let $g \in H^{1/2,1/4}(\Sigma)$ or $h \in H^{-1/2,-1/4}(\Sigma)$. Then the initial boundary value problem \eqref{eq:heat_equation}--\eqref{eq:initial_condition} with Dirichlet boundary condition \eqref{eq:dirichlet_boundary_values} or Neumann boundary condition~\eqref{eq:neumann_boundary_values} admits a unique solution $u \in H^{1,1/2}_{;0,}(Q)$.
\end{theorem}

A possible strategy to determine the solution of the initial boundary value problem of the heat equation is to consider related boundary integral equations. The solution $u$ in Theorem~\ref{thm:unique_solvability_heat_equation} satisfies the representation formula \cite[cf.~Theorem~2.20]{Cos1990}
\begin{equation} \label{eq:representation_formula}
  u = - W \dtrace{u} + \widetilde V (\alpha \ntrace{u}),
\end{equation}
with the single layer potential $\widetilde{V}: H^{-1/2,-1/4}(\Sigma) \rightarrow H^{1,1/2}_{;0,}(Q)$ defined by
\begin{equation} \label{eq:def_slp}
  \widetilde{V} q = G_\alpha \ast \left( (\dtrace)'q \right)
\end{equation}
and the double layer potential $W: H^{1/2,1/4}(\Sigma) \rightarrow H^{1,1/2}_{;0,}(Q)$ given by
\begin{equation} \label{eq:def_dlp}
  W \varphi = G_\alpha \ast \left( (\alpha \ntrace)'\varphi \right).
\end{equation}
Here, $(\dtrace)'$ and $(\ntrace)'$ are the adjoint trace operators defined in \eqref{eq:adjoint_dtrace} and \eqref{eq:adjoint_ntrace}, respectively, $G_\alpha$ is the fundamental solution of the heat equation
\begin{equation} \label{eq:def_fundamental_solution}
  G_\alpha(\vx,t) :=
  \begin{cases}
  \displaystyle
  \frac{1}{(4\pi\alpha t)^{3/2}}\exp\bigg( -\frac{|\vx|^2}{4\alpha t} \bigg) & \text{for } t > 0, \\
  0 & \text{otherwise},
  \end{cases}
\end{equation}
and the convolution is understood as convolution of distributions defined in \eqref{eq:def_convolution_distributions}, where we identify $G_\alpha$ with the regular distribution induced by it. For~$(x,t) \in (\mathbb{R}^3 \times (0,T)) \backslash \Sigma$ we have the usual representations of the single and double layer potentials as
\begin{align*}
  \widetilde{V} q ( \vx, t ) &= \int_{0}^t \int_{\Gamma} 
    G_\alpha(\vx-\vy, t - \tau) q(\vy, \tau) \, \dif\vs_\vy\, \dif\tau, \\
  W \varphi ( \vx, t ) &= \int_{0}^t \int_{\Gamma} 
    \alpha \frac{\partial G_\alpha}{\partial \vn_\vy}(\vx-\vy, t - \tau) \varphi(\vy, \tau)\, \dif\vs_\vy\, \dif\tau.
\end{align*}
The latter integral is well-defined for all $\varphi \in H^{1/2,1/4}(\Sigma)$, while for general $q \in H^{-1/2,-1/4}(\Sigma)$ the integral representation of $\widetilde{V}q$ has to be understood in the sense of the duality pairing $\langle \cdot, \cdot \rangle_\Sigma$.

By applying the Dirichlet and Neumann trace operators to the potentials $\widetilde{V}$ and $W$ we get the usual boundary integral operators, i.e.~the single layer operator $V$, double layer operator $K$, adjoint time reversed double layer operator $K_T'$ and hypersingular operator $D$ defined by
\begin{alignat*}{2}
  V&: H^{-1/2,-1/4}(\Sigma) \rightarrow H^{1/2,1/4}(\Sigma), \qquad 
    && V q := \dtrace \widetilde{V}q, \\
  K&: H^{1/2,1/4}(\Sigma) \rightarrow H^{1/2,1/4}(\Sigma), \qquad 
    && K \varphi := \frac{1}{2}\left( \dtrace W \varphi + \dtraceext W \varphi \right),  \\
  K_T'&: H^{-1/2,-1/4}(\Sigma) \rightarrow H^{-1/2,-1/4}(\Sigma), \qquad 
    && K_T' q := \frac{1}{2}\left( \alpha \ntrace \widetilde{V} q + \alpha \ntraceext \widetilde{V} q  \right), \\
  D&: H^{1/2,1/4}(\Sigma) \rightarrow H^{-1/2,-1/4}(\Sigma), \qquad 
    && D \varphi := \alpha \ntrace W \varphi, \\
\end{alignat*}
see \cite[Defintion~3.5]{Cos1990}. In that paper the interested reader can also find a proper definition of the exterior traces $\dtraceext$ and $\ntraceext$, which is skipped here because these operators are not considered anymore in the following. Regarding the notation and naming of~$K_T'$ we want to point out that $K_T'$ is not the adjoint of $K$, but the adjoint of the time reversed double layer operator $K_T := \Theta_T \circ K$, with $\Theta_T f (\vx,t):=f(\vx,T-t)$.

The boundary integral operators can be used in the usual way to obtain solutions of the initial boundary value problems \eqref{eq:heat_equation}--\eqref{eq:boundary_condition}, see e.g.~\cite{Cos1990,DohNiiSte2019}. Our focus in this paper lies on the hypersingular operator~$D$ and in particular its application. Before dealing with that, we conclude this section with Theorem~\ref{thm:extended_maximum_principle}, which is a generalization of the maximum principle in Proposition~\ref{prop:maximum_principle} and will be required in the proof of Theorem~\ref{thm:second_bilinear_form_tensor_case}. For Theorem~\ref{thm:extended_maximum_principle} in turn, we have to state yet another solvability result, namely Theorem~\ref{thm:heat_solvability_l2}.
\begin{remark}
  In the following theorem non-tangential limits and the non-tangential maximal function $N(u)$ of a function~$u$ in~$Q$ will appear. We will use that
  \begin{equation} \label{eq:definition_and_estimate_nontangential_max}
    N(u)(\vx,t) := \sup \{ |u(\vy,\tau)| : (\vy,\tau) \in \gamma(\vx,t)\} \leq \sup\{ |u(\vy,\tau)|: (\vy,\tau) \in Q \}
  \end{equation}
  for all~$(\vx,t) \in \Sigma$, where $\gamma(\vx,t) \subset Q$ is the so-called parabolic non-tangential approach region of~$(\vx,t) \in \Sigma$. The non-tangential limit of a function $u$ in $Q$ at $(\vx,t) \in \Sigma$ is also defined with respect to this region as 
  \begin{equation*}
    \lim_{\gamma(\vx,t)\ni(\vy,\tau) \rightarrow (\vx,t)} u(\vy,\tau).
  \end{equation*}
  For the later discussion we do not need the concrete definition of $\gamma(\vx,t)$, so we refer to \cite[Section~1]{Bro1989} for it.
\end{remark}
\begin{theorem}[\!\!{\cite[Theorems~8.1~and~8.3]{Bro1989}}] \label{thm:heat_solvability_l2}
  The operator $(-1/2\, I + K)$ is an isometry from~$L^2(\Sigma)$ to $L^2(\Sigma)$. In particular, for any $g \in L^2(\Sigma)$ the function 
  \begin{equation} \label{eq:indirect_dl_approach}
    u = W \bigg( \left(-\frac{1}{2}I + K\right)^{-1} g \bigg)
  \end{equation}
  is well-defined. Furthermore, it is the unique function satisfying the following properties:
  \begin{enumerate}[(i)]
    \item $u \in C^\infty(Q)$ and $(\partial/\partial t - \alpha \upDelta)u = 0$,
    \item $u \in C(\Omega \times [0,T))$ and $u(\cdot,0) = 0$,
    \item the non-tangential maximal function $N(u)$ is in $L^2(\Sigma)$,
    \item $u = g$ on~$\Sigma$ in the sense of non-tangential limits almost everywhere.
  \end{enumerate} 
\end{theorem}

\begin{theorem}[Extended parabolic maximum principle] \label{thm:extended_maximum_principle}
  Let $g \in L^\infty(\Sigma)$ and $u$ be the solution to the initial Dirichlet boundary value problem \eqref{eq:heat_equation}--\eqref{eq:dirichlet_boundary_values} given by~\eqref{eq:indirect_dl_approach}. Then there holds 
  \begin{equation} \label{eq:max_at_boundary_l_infty}
    \sup\{ |u(\vx,t)| : (\vx,t) \in Q \} \leq \|g\|_{L^\infty(\Sigma)}.
  \end{equation}
\end{theorem}

\begin{proof}
The idea is to approximate the boundary datum $g$ in $L^2(\Sigma)$ by a sequence $\{g_n\}_n$ in $C(\overline{\Sigma})$ such that $g_n(\vx,0) = 0$ for all $\vx \in \Gamma$ and $\|g_n\|_{L^\infty(\Sigma)} \leq \|g\|_{L^\infty(\Sigma)}$, to construct solutions $u_n$ of the homogeneous heat equation by Theorems~\ref{thm:classical_solvability_heat_equation} and~\ref{thm:heat_solvability_l2} such that $u_n = g_n$ on $\Sigma$, and to show~\eqref{eq:max_at_boundary_l_infty} by contradiction using Proposition~\ref{prop:maximum_principle} for $u_n$ and a continuity argument.

We start with the construction of the sequence $\{g_n\}_n$. Due to the density of $C(\overline{\Sigma})$ in $L^2(\Sigma)$ we find a sequence $\{f_n\}_n$ in $C(\overline{\Sigma})$ such that $f_n \rightarrow g$ in $L^2(\Sigma)$. Let us first define $\widetilde{g}_n$ by
\begin{equation*}
  \widetilde{g}_n(\vx,t) := \begin{cases}
    f_n(\vx,t), \qquad &\text{if } (\vx,t) \text{ is such that } |f_n(\vx,t)|\leq \|g\|_{L^{\infty}(\Sigma)}, \\
    \sign(f_n(\vx,t)) \|g\|_{L^\infty(\Sigma)}, &\text{otherwise}.
  \end{cases}
\end{equation*}
It is easy to see that $\widetilde{g}_n \in C(\overline{\Sigma})$ and $\|\widetilde{g}_n\|_{L^\infty(\Sigma)} \leq \|g\|_{L^\infty(\Sigma)}$ for all $n\in \mathbb{N}$. In addition, $\widetilde{g}_n \rightarrow g$ in~$L^2(\Sigma)$, which follows from the estimate
\begin{equation*}
  |\widetilde{g}_n(\vx,t) - g(\vx,t)| \leq |f_n(\vx,t) - g(\vx,t)|
\end{equation*}  
for all $n \in \mathbb{N}$ and almost all $(\vx,t) \in \Sigma$. Since $\widetilde{g}_n(\vx,0)=0$ might be violated for some $\vx \in \Gamma$ we~set
\begin{equation*}
  g_n(\vx,t) := \begin{cases}
    n\, t\, \widetilde{g}_n(\vx,t), \qquad &\text{if } 0 \leq t \leq \frac{1}{n}, \\
    \widetilde{g}_n(\vx,t), \qquad &\text{otherwise.}
  \end{cases}
\end{equation*} 
By construction, there holds $g_n \in C(\overline{\Sigma})$, $g_n(\cdot,0)=0$ on $\Gamma$ and $\|g_n\|_{L^\infty(\Sigma)} \leq \|g\|_{L^\infty(\Sigma)}$ for all~$n \in \mathbb{N}$. In addition, $g_n - \widetilde{g}_n \rightarrow 0$ in $L^2(\Sigma)$ which implies the convergence of $g_n$ to $g$ in $L^2(\Sigma)$. Therefore, $\{g_n\}_n$ is a sequence in $C(\overline{\Sigma})$ with the desired properties.

By Theorem~\ref{thm:classical_solvability_heat_equation}, we can find for each $n \in \mathbb{N}$ a unique $u_n \in C^\infty(Q) \cap C(\overline{Q})$ which solves the heat equation \eqref{eq:heat_equation} and \eqref{eq:initial_condition} and satisfies $u_n = g_n$ on $\Sigma$. From the classical parabolic maximum principle in Proposition~\ref{prop:maximum_principle} it follows that 
\begin{equation} \label{eq:estimate_maximum_u_n}
  \sup\{|u_n(\vx,t)| : (\vx,t) \in Q\} \leq \|g_n\|_{L^\infty(\Sigma)} \leq \|g\|_{L^\infty(\Sigma)}.
\end{equation}
Together with \eqref{eq:definition_and_estimate_nontangential_max} this yields $N(u_n) \leq \|g\|_{L^\infty(\Sigma)}$ on $\Sigma$ and thus $N(u_n) \in L^2(\Sigma)$. Hence, $u_n$ satisfies all properties of Theorem~\ref{thm:heat_solvability_l2} and we obtain the representation
\begin{equation*}
  u_n = W \bigg( \left(-\frac{1}{2}I + K\right)^{-1} g_n \bigg).
\end{equation*}
With this representation we can show that $u_n \rightarrow u$ locally in $Q$ in $L^2$. In fact, let $Q_\varepsilon := \{(\vx,t) \in Q : \dist(\vx, \Gamma) > \varepsilon \}$. Then the convergence of~$u_n$ to~$u$ in $L^2(Q_\varepsilon)$ follows immediately from $g_n \rightarrow g$ in $L^2(\Sigma)$, since $(-\frac{1}{2}I + K)$ is an isomorphism in $L^2(\Sigma)$ as stated in Theorem~\ref{thm:heat_solvability_l2} and $W: L^2(\Sigma) \rightarrow L^2(Q_\varepsilon)$ is continuous, which is easy to see.

Suppose now that \eqref{eq:max_at_boundary_l_infty} does not hold true. Then there exists a space-time point $(\vx_0,t_0) \in Q$ such that $|u(\vx_0,t_0)| > \|g\|_{L^{\infty}(\Sigma)}$. Since $u$ is continuous in $Q$ due to Theorem~\ref{thm:heat_solvability_l2}~\emph{(ii)}, we can find some $\varepsilon > 0$, $\delta > 0$ and an open set $A \subset Q_\varepsilon$ with measure $|A|>0$ such that $(\vx_0,t_0) \in A$ and $|u(\vx,t)|> \|g\|_{L^{\infty}(\Sigma)} + \delta$ for all $(\vx,t) \in A$. Together with \eqref{eq:estimate_maximum_u_n} it follows that
\begin{equation*}
  \iint_{Q_\varepsilon} |u_n(\vx,t) - u(\vx,t)|^2\,\dif\vx\,\dif t
  \geq \iint_A |u_n(\vx,t) - u(\vx,t)|^2\,\dif\vx\,\dif t > \delta^2 |A|
\end{equation*}
for all $n \in \mathbb{N}$, which is a contradiction to $u_n \rightarrow u$ in $L^2(Q_\varepsilon)$. Therefore, \eqref{eq:max_at_boundary_l_infty} is satisfied.
\end{proof}

\section{An integration by parts formula for the evaluation of \texorpdfstring{$\langle D \cdot, \cdot \rangle$}{the bilinear form of D}} \label{sec:ibp_formula}
The hypersingular boundary integral operator $D: H^{1/2,1/4} \rightarrow H^{-1/2,-1/4}(\Sigma)$ was introduced in Section~\ref{sec:heat_equation} as $-\alpha \ntrace W$, so formally it is given by
\begin{equation*}
  D q (\vx, t) = - \alpha \frac{\partial}{\partial \vn_\vx} \int_{0}^t \int_{\Gamma} 
    \alpha \frac{\partial G_\alpha}{\partial \vn_\vy}(\vx-\vy, t - \tau) q(\vy, \tau)\, \dif\vs_\vy\, \dif\tau.
\end{equation*}
A major difficulty when dealing with this operator is to find an explicit representation for its evaluation. Since the kernel $((\vx,t),(\vy,\tau)) \mapsto \partial/\partial \vn_\vx \partial/\partial \vn_\vy G_\alpha (\vx-\vy,t-\tau)$ is not integrable in a vicinity of the diagonal $(\vx,t)=(\vy,\tau)$ one cannot simply exchange the order of differentiation and integration in the above formula. Often however, one does not consider the operator $D$ itself but the associated bilinear form~$\langle D \cdot, \cdot \rangle_\Sigma$ on $H^{1/2,1/4}(\Sigma) \times H^{1/2,1/4}(\Sigma)$. For this bilinear form an alternative representation via integration by parts is available, which eventually allows for an evaluation by means of weakly singular integrals. A general form of this representation formula is provided in the following theorem. Its proof is one of the main results in this paper.
\begin{theorem} \label{thm:general_ibp_formula}
  Let $\Gamma$ be the boundary of a bounded Lipschitz domain and $\Sigma = \Gamma \times (0,T)$. For~$\varphi, \psi \in H^{1/2,1/4}(\Sigma)$ there holds the integration by parts formula
  \begin{equation} \label{eq:general_ibp_formula}
    \langle D \varphi, \psi \rangle_\Sigma = 
      \alpha^2 \langle \scurl \psi, V (\scurl \varphi) \rangle_{\Sigma} + \alpha \, b(\varphi, \psi).
  \end{equation}
  Here, the single layer boundary integral operator $V$ is applied componentwise to $\scurl \varphi$ and the bilinear form $b(\cdot,\cdot): H^{1/2,1/4}(\Sigma) \times H^{1/2,1/4}(\Sigma) \rightarrow \mathbb{R}$ is defined by
  \begin{equation} \label{eq:def_b}
    b(\varphi, \psi) 
    := \left( \frac{\partial}{\partial t} (\dtrace)' \Big( V(\varphi \vn) \cdot \vn \Big) \right)[\widetilde{\psi}]
    :=- \left\langle V (\varphi \vn), \frac{\partial}{\partial t} \psi \vn \right\rangle_{\Sigma},
  \end{equation}
  for $\varphi \in H^{1/2,1/4}(\Sigma)$, $\psi \in \dtrace(C^\infty_c(\mathbb{R}^3\times(0,T)))$ and $\widetilde{\psi} \in C^\infty_c(\mathbb{R}^3 \times (0,T))$ such that $\psi = \widetilde{\psi}|_{\Sigma}$, and as its continuous extension for general $\psi \in H^{1/2,1/4}(\Sigma)$.  
\end{theorem}

We will give a rigorous proof of this theorem in Section~\ref{subsec:proof_general_ibp_formula}. A corresponding result for the~2D case has been given in \cite[Theorem 6.1]{Cos1990}, including an outline of the proof. In that paper the bilinear form $b$ is represented by $b{(\varphi, \psi) = \langle \partial/\partial t\, V (\varphi \vn), \psi \vn \rangle_\Sigma}$ and it is stated in the proof that it has to be interpreted in the sense of a continuous extension. For the 3D case a similar statement can be found in \cite[Section~4.7]{Doh2019} and \cite[Section~3.1.3]{Mes2014}, but no proof is given. In addition, the latter authors formulate the result in a less rigorous way and do not clarify how the second term on the right-hand side of \eqref{eq:general_ibp_formula}, which they represent as 
\begin{equation} \label{eq:formal_def_second_bilinear_form}
  - \alpha \int_{0}^T \int_{\Gamma} \psi(\vx, t) \vn(\vx)\cdot \int_{0}^t \int_{\Gamma} 
    \frac{\partial G_{\alpha}}{\partial\tau}(\vx-\vy, t-\tau) \varphi(\vy, \tau) \vn(\vy) 
    \, \dif\vs_\vy \, \dif\tau \, \dif\vs_\vx \, \dif t,
\end{equation} 
has to be understood for general $\varphi$ and $\psi$. In Proposition~\ref{prop:kernel_not_integrable} we will show that the kernel of this bilinear form, i.e.~the function $((\vx,t),(\vy,\tau)) \mapsto \partial / \partial \tau\, G_\alpha(\vx-\vy, t-\tau)$ is not Lebesgue integrable on~$\Sigma \times \Sigma$. This makes it even more difficult to give a suitable meaning to representation~\eqref{eq:formal_def_second_bilinear_form}. 

Since we define the bilinear form $b$ as continuous extension of \eqref{eq:def_b} it is a priori not clear, how to evaluate it for nonsmooth $\psi$. The reason is that neither $V (\varphi \vn)$ nor $\psi$ does admit a weak derivative with respect to time in general, which is why the second term in \eqref{eq:def_b} has to be understood in the stated distributional sense for smooth $\psi$ as above. In Theorem~\ref{thm:second_bilinear_form_tensor_case} we present an alternative representation of $b$ valid for certain classes of functions, which overcomes this deficiency. 
\subsection{A proof of the general integration by parts formula} \label{subsec:proof_general_ibp_formula}
The proof of Theorem~\ref{thm:general_ibp_formula} is split into three main steps, to each of which we dedicate a separate paragraph. In the first paragraph we derive an alternative representation of $\alpha \nabla W \varphi$. The steps in the second and third paragraph are based on the ideas given in \cite[Proof of Theorem~6.1]{Cos1990}. We show an integration by parts formula on an auxiliary boundary inside of the space-time domain~$Q$ in the second paragraph, using the representation of $\alpha \nabla W \varphi$ from the first one. In the third paragraph we construct a sequence of auxiliary boundaries~$\Sigma_m$ inside of $Q$ which approximate~$\Sigma$ and show that the integration by parts formula on $\Sigma$ is obtained from the formulas on $\Sigma_m$ in the limit as $m$ tends to infinity. The actual proof of Theorem~\ref{thm:general_ibp_formula} is given at the end of the third paragraph. 

\subsubsection*{An alternative representation of the gradient of the double layer potential}

The double layer potential $W\varphi$ of a function $\varphi \in H^{1/2,1/4}(\Sigma)$ is given by \eqref{eq:def_dlp}, i.e.~the convolution of the fundamental solution $G_\alpha$ in \eqref{eq:def_fundamental_solution} with the distribution $(\alpha \ntrace)'\varphi$. We use this definition to derive an alternative representation of its distributional gradient. This approach is motivated by the proof of the integration by parts formula for elliptic operators in \cite[Section~3.3.4]{SauSch2010}.
\begin{proposition} \label{prop:distributional_representation_grad_dlp}
  Let $\varphi \in H^{1/2,1/4}(\Sigma)$ and $u=W\varphi$ be defined by \eqref{eq:def_dlp}. Then there holds
  \begin{equation} \label{eq:distributional_representation_grad_dlp}
    \alpha \nabla u
    = - G_\alpha \ast \left( \alpha \curl \curl (\vdtrace)'(\alpha \varphi \vn) 
        + \frac{\partial}{\partial t} (\vdtrace)'(\alpha \varphi \vn) \right)
      + (\dtrace)'(\alpha \varphi \vn)
  \end{equation}
  in the distributional sense on $\mathbb{R}^3 \times \mathbb{R}$, where the convolution is understood componentwise. In particular, inside of $Q$ there holds
  \begin{equation} \label{eq:representation_grad_dlp_in_Q}
    \alpha \nabla u 
    = - \left( \alpha \curl \curl (\vdtrace)'(\alpha \varphi \vn) \right) \ast G_\alpha 
      - \left( \frac{\partial}{\partial t} (\vdtrace)'(\alpha \varphi \vn) \right) \ast G_\alpha
  \end{equation}
  in the classical sense, and the terms on the right-hand side can be understood as functions in~$C^\infty(Q)$.
\end{proposition}

\begin{proof}
  The double layer potential $u=W\varphi$ is defined in \eqref{eq:def_dlp} as the convolution of two distributions and therefore it is itself a distribution on $\mathbb{R}^3 \times \mathbb{R}$. We want to take the distributional gradient of it. For this purpose, let $\vv \in \mathbf{C}^\infty_c(\mathbb{R}^3 \times \mathbb{R})$. Then there holds
  \begin{equation} \label{eq:distributional_gradient_dlp}
    \begin{split}
      \alpha \nabla u \, [\vv] &= - \alpha u[\diver \vv] 
        = - \alpha \Big( G_\alpha \ast ((\alpha \ntrace)'\varphi)\Big) [\diver \vv] \\
        &= -\alpha (G_\alpha)_{(\vx,t)}\left[ ((\alpha \ntrace)'\varphi)_{(\vy,\tau)} [ 
          \diver \vv (\cdot_{\vx} + \cdot_{\vy}, \cdot_{t} + \cdot_{\tau}) ]  \right],
    \end{split}
  \end{equation}
  where we used the definition of the convolution of distributions in \eqref{eq:def_convolution_distributions}. With the definition of~$(\ntrace)'$ in \eqref{eq:adjoint_ntrace} and $(\dtrace)'$ in \eqref{eq:adjoint_dtrace} it follows that
  \begin{align*}
    ((\alpha \ntrace)'\varphi&)_{(\vy,\tau)} [ \diver \vv (\vx + \cdot_{\vy}, t + \cdot_{\tau}) ] \\
    &= \langle ( \alpha \ntrace )_{\vy,\tau} (\diver \vv(\vx+\cdot_\vy, t+\cdot_\tau)), \varphi \rangle_\Sigma \\
    &= \int_0^T\int_\Gamma \varphi(\vy,\tau)\, \alpha \vn(\vy)\cdot \nabla \diver \vv (\vx+\vy, t+\tau) \, \dif\vs_{\vy}\, \dif \tau \\
    &= ((\dtrace)'(\alpha \varphi \vn))_{(\vy,\tau)} [\nabla \diver \vv (\vx + \cdot_{\vy},t + \cdot_{\tau})]
  \end{align*}
  for each $(\vx,t) \in \mathbb{R}^3 \times \mathbb{R}$. By inserting this into \eqref{eq:distributional_gradient_dlp} and using ${\nabla \diver \vv = \curl \curl \vv + \upDelta \vv}$, where $\upDelta$ is applied componentwise to $\vv$, we get
  \begin{align*}
    \alpha \nabla u \, [\vv] &= -\alpha(G_\alpha)_{(\vx,t)}\left[ ((\dtrace)'(\alpha \varphi \vn))_{(\vy,\tau)} [ 
      (\curl\curl + \upDelta) \vv (\cdot_{\vx} + \cdot_{\vy}, \cdot_{t} + \cdot_{\tau}) ]  \right] \\
      &= -\alpha \Big( G_\alpha \ast ((\dtrace)'(\alpha \varphi \vn))\Big) [\curl\curl \vv + \upDelta \vv] \\
      &= -\alpha \curl\curl \Big(G_\alpha \ast ((\dtrace)'(\alpha \varphi \vn))\Big) [\vv]
        -\alpha \upDelta \Big( G_\alpha \ast ((\dtrace)'(\alpha \varphi \vn))\Big) [\vv].
  \end{align*}
  Here and in the following the convolution of $G_\alpha$ and a vector valued distribution like $(\dtrace)'(\alpha \varphi \vn)$ is understood componentwise. By differentiation rules for convolutions of distributions, which follow from the one stated in Proposition~\ref{prop:properties_convolution_distribution}, it follows 
  \begin{align*}
    -\alpha \curl\curl \Big(G_\alpha \ast ((\dtrace)'(\alpha \varphi \vn))\Big) [\vv] &=  
      -\Big(G_\alpha \ast \Big(\alpha \curl\curl (\dtrace)'(\alpha \varphi \vn)\Big)\Big) [\vv], \\
    -\alpha \upDelta \Big( G_\alpha \ast ((\dtrace)'(\alpha \varphi \vn))\Big) [\vv] &=
      \Big((-\alpha \upDelta G_\alpha) \ast ((\dtrace)'(\alpha \varphi \vn))\Big) [\vv].
  \end{align*}
  Since $G_\alpha$ is a fundamental solution of the heat equation~\eqref{eq:heat_equation} there holds ${\partial/\partial t\, G_\alpha-\alpha \upDelta G_\alpha = \delta_\vzero}$, where~$\delta_\vzero$ denotes the delta distribution concentrated at $\vzero$. As a consequence we can rewrite the second equation as
  \begin{align*}
    \Big((-\alpha \upDelta G_\alpha) \ast ((\dtrace)'(\alpha \varphi \vn))\Big) [\vv] &=
      \left( \left(-\frac{\partial}{\partial t} G_\alpha + \delta_\vzero  \right) 
        \ast ((\dtrace)'(\alpha \varphi \vn))\right) [\vv] \\
      &= - \left(G_\alpha \ast \bigg(\frac{\partial}{\partial t} (\dtrace)'(\alpha \varphi \vn)\bigg)\right) [\vv]
      + (\dtrace)'(\alpha \varphi \vn) [\vv].
  \end{align*}
  Collecting all results we see that
  \begin{align*}
    \alpha \nabla u \, [\vv] = - \left( G_\alpha \ast \left( \alpha \curl \curl (\vdtrace)'(\alpha \varphi \vn)  
      + \frac{\partial}{\partial t} (\vdtrace)'(\alpha \varphi \vn) \right) \right) [\vv]
      + (\dtrace)'(\alpha \varphi \vn)[\vv].
  \end{align*}
  Since $\vv \in C^\infty_c(\mathbb{R}^3 \times \mathbb{R})$ was arbitrary, we conclude that \eqref{eq:distributional_representation_grad_dlp} holds. 

  Let us now interpret both sides of \eqref{eq:distributional_representation_grad_dlp} as elements in $\mathcal{D}'(Q)$ by restricting the test functions to $\mathbf{C}^\infty_c(Q)$. Since the support of $(\dtrace)'(\alpha \varphi \vn)$ is a subset of $\overline{\Sigma}$ there holds 
  \begin{equation*}
    \alpha \nabla u \, [\vv] = - \left( G_\alpha \ast \left( \alpha \curl \curl (\vdtrace)'(\alpha \varphi \vn)  
      + \frac{\partial}{\partial t} (\vdtrace)'(\alpha \varphi \vn) \right) \right) [\vv]
  \end{equation*}
  for all $\vv \in \mathbf{C}^\infty_c(Q)$. The left-hand side can be interpreted as a regular distribution on $Q$ induced by $\alpha \nabla u \in \mathbf{L}^2(Q)$, because $u \in H^{1,1/2}_{;0,}(Q)$. The first term on the right-hand side can be rewritten as
  \begin{equation*}
    G_\alpha \ast \Big( \alpha \curl \curl (\vdtrace)'(\alpha \varphi \vn) \Big) 
      = \Big( \alpha \curl \curl (\vdtrace)'(\alpha \varphi \vn) \Big) \ast G_\alpha.
  \end{equation*}
  \ifshort
   Since ${G_\alpha \in C^\infty((\mathbb{R}^3 \times \mathbb{R}) \backslash \{0\})}$ and the distribution $\curl \curl (\vdtrace)'(\alpha \varphi \vn)$ has support in $\overline{\Sigma}$, one can show that $(\curl \curl (\vdtrace)'(\alpha \varphi \vn) ) \ast G_\alpha$ is a regular distribution on $Q$ and can be interpreted as a function in $\mathbf{C}^\infty(Q)$. 
  \fi
  The same arguments apply to 
  \begin{equation*}
    G_\alpha \ast \left( \frac{\partial}{\partial t} (\vdtrace)'(\alpha \varphi \vn) \right)
      = \left( \frac{\partial}{\partial t} (\vdtrace)'(\alpha \varphi \vn) \right) \ast G_\alpha. \qedhere
  \end{equation*}
\end{proof}

For later reference we define $\vg_1, \vg_2 \in \mathbf{C}^\infty(Q)$ in agreement with Proposition~\ref{prop:distributional_representation_grad_dlp} by
\begin{align}
  \vg_1 &:= -\left( \alpha \curl \curl (\vdtrace)'(\alpha \varphi \vn) \right) \ast G_\alpha, 
  \label{eq:def_g1} \\
  \vg_2 &:= -\left( \frac{\partial}{\partial t} (\vdtrace)'(\alpha \varphi \vn) \right) \ast G_\alpha.
  \label{eq:def_g2}
\end{align}
Equation~\eqref{eq:representation_grad_dlp_in_Q} can then be written in the short form
\begin{equation} \label{eq:representation_grad_dlp_short}
  \alpha \nabla u = \vg_1 + \vg_2 \qquad \text{in } Q.  
\end{equation}

\subsubsection*{An integration by parts formula inside the space-time domain $Q$}

The next major step in the proof of Theorem~\ref{thm:general_ibp_formula} is to show a result equivalent to the integration by parts formula \eqref{eq:general_ibp_formula} on an auxiliary boundary $\Sigma_m$ inside of $Q$. This is formulated in Proposition~\ref{prop:interior_ibp_formula}. The approach is motivated by the smoothness of $u = W \varphi$ or rather $\alpha \nabla u$ inside of $Q$, which permits us to represent normal derivatives of $u$ on $\Sigma_m$ in a classical way. Furthermore it allows us to interpret derivatives appearing in duality products $\langle \cdot, \cdot \rangle_{\Sigma_m}$ in a distributional sense and thus to integrate by parts.

\begin{proposition} \label{prop:interior_ibp_formula}
  Let $\Omega_m$ be a Lipschitz domain satisfying ${\overline{\Omega}_m \subset \Omega}$ with outward normal vector~$\nvecm$ and boundary $\Gamma_m := \partial \Omega_m$. Let $Q_m := \Omega_m \times (0,T)$ and $\Sigma_m := \Gamma_m \times (0,T)$. Let $\varphi \in H^{1/2,1/4}(\Sigma)$, $u = W\varphi$ and $\psi_m \in H^{1/2,1/4}(\Sigma_m)$. Then
  \begin{equation} \label{eq:interior_ibp_formula}
    \begin{split}
      -\alpha\,\langle \ntracem u, \psi_m \rangle_{\Sigma_m} &= 
        \alpha^2 \langle \scurlm \psi_m, \dtracem \widetilde{V}(\curl_{\Sigma} \varphi) \rangle_{\Sigma_m} \\
        &+ \alpha \, \left\langle \nvecm\cdot 
          \dtracem\left(\frac{\partial}{\partial t} \widetilde{V} (\varphi \vn ) \right), 
          \psi_m \right\rangle_{\Sigma_m}.
    \end{split}
  \end{equation}
\end{proposition}

Before proving Proposition~\ref{prop:interior_ibp_formula} we consider two lemmata. Keeping in mind the decomposition~\eqref{eq:representation_grad_dlp_short}, we focus first on $\vg_1$ in \eqref{eq:def_g1}. In Lemma~\ref{lem:first_part_interior_ibp_formula} we show that it is related to the first term on the right-hand side of~\eqref{eq:interior_ibp_formula}. For this purpose we have to draw a connection between $\curl (\dtrace)'(\varphi \vn)$ and the surface curl of $\varphi$ defined in \eqref{eq:def_scurl}. This is done in Lemma~\ref{lem:curl_and_scurl} whose proof is given in Section~\ref{sec:appendix_proofs}. A similar result for purely spatial curls is proven in \cite[cf.~Lemma~3.3.21]{SauSch2010}.

\begin{lemma} \label{lem:curl_and_scurl}
  Let $\varphi \in H^{1/2,1/4}(\Sigma)$. Then 
  \begin{equation} \label{eq:curl_and_scurl}
    \curl ( \dtrace )'(\varphi \vn ) = (\dtrace)' (\scurl \varphi)
  \end{equation}
  in $\mathcal{D}'(\mathbb{R}^3 \times \mathbb{R})$. 
\end{lemma}

\begin{lemma} \label{lem:first_part_interior_ibp_formula}
  Let $\Sigma_m$ be given as in Proposition~\ref{prop:interior_ibp_formula}. Let $\varphi \in H^{1/2,1/4}(\Sigma)$, $\psi_m \in H^{1/2,1/4}(\Sigma_m)$ and~$\vg_1$ be defined by \eqref{eq:def_g1}. Then
  \begin{equation} \label{eq:first_part_interior_ibp_formula}
    -\langle \nvecm\cdot \dtracem \vg_1, \psi_m \rangle_{\Sigma_m} 
    = \alpha^2 \langle \scurlm \psi_m, \dtracem \widetilde{V}(\curl_{\Sigma} \varphi) \rangle_{\Sigma_m}.
  \end{equation}
\end{lemma}

\begin{proof}
  Since $\vg_1 \in \mathbf{C}^\infty(Q)$ we can interpret 
  \begin{equation*}
    -\langle \nvecm\cdot  \dtracem \vg_1, \psi_m \rangle_{\Sigma_m} 
    = -\int_0^T \int_{\Gamma_m} \psi_m(\vx,t) \nvecm(\vx) \cdot \vg_1(\vx,t)\,\dif\vs_\vx\,\dif t
    = -(\dtracem)'(\psi_m \nvecm) [ \vg_1 ],
  \end{equation*}
  with $(\dtracem)'(\psi_m \nvecm) \in \boldsymbol{\mathcal{E}}'(Q)$. By \eqref{eq:def_g1} we have
  \begin{equation*}
    \vg_1 = -\left( \alpha \curl \curl (\vdtrace)'(\alpha \varphi \vn )\right) \ast G_\alpha
      = -\alpha^2 \curl \left( \left( \curl (\vdtrace)'(\varphi \vn)\right) \ast G_\alpha \right)
  \end{equation*}
  and we see that 
  \begin{align*}
    - (\dtracem)'(\psi_m \nvecm) [ \vg_1 ]
      &= \alpha^2 (\dtracem)'(\psi_m \nvecm) 
        \left[ \curl \left( \left( \curl (\vdtrace)'(\varphi \vn)\right) \ast G_\alpha \right) \right] \\
      &= \alpha^2  \curl (\dtracem)'(\psi_m \nvecm)
        \left[\left( \curl (\vdtrace)'(\varphi \vn) \right) \ast G_\alpha \right],
  \end{align*}
  i.e.~we can integrate by parts in a distributional sense in the duality pairing of $\mathcal{E}'(Q)$ and $C^\infty(Q)$. The identity \eqref{eq:curl_and_scurl}, which obviously still holds if $\Sigma$ is replaced by $\Sigma_m$, combined with the previous equations yields
  \begin{align*}
    -\langle \nvecm\cdot  \dtracem \vg_1, \psi_m \rangle_{\Sigma_m} 
      &= \alpha^2 ( \dtracem )' ( \scurlm \psi_m ) 
        \left[ \left( (\dtrace)' (\scurl \varphi) \right) \ast G_\alpha \right]\\
      &= \alpha^2 ( \dtracem )' ( \scurlm \psi_m ) \left[ \widetilde{V} (\scurl \varphi) \right] \\
      &= \alpha^2 \langle \scurlm \psi_m, \dtracem \widetilde{V} (\scurl \varphi) \rangle_{\Sigma_m},
  \end{align*}
  where we used the definition of $\widetilde{V}$ in \eqref{eq:def_slp} in the second line and understand its application here in a componentwise way.
\end{proof}

\begin{proof}[Proof of Proposition~\ref{prop:interior_ibp_formula}]  
  We have seen in \eqref{eq:representation_grad_dlp_short} that the scaled gradient $\alpha \nabla u$ can be written as the sum of the functions $\vg_1, \vg_2 \in C^\infty(Q)$ given in \eqref{eq:def_g1} and \eqref{eq:def_g2}. As a consequence, the scaled Neumann trace $\alpha \ntracem u$ on the boundary $\Sigma_m$ inside of $Q$ is simply given by 
  \begin{equation*}
    \alpha \ntracem u = \alpha \nvecm \cdot \dtracem(\nabla u) = \nvecm \cdot \dtracem \vg_1  + \nvecm \cdot \dtracem \vg_2
  \end{equation*}
  and hence
  \begin{equation*}
    -\alpha\,\langle \ntracem u, \psi_m \rangle_{\Sigma_m} 
      = -\langle \nvecm \cdot \dtracem \vg_1, \psi_m \rangle_{\Sigma_m} 
        -\langle \nvecm \cdot \dtracem \vg_2, \psi_m \rangle_{\Sigma_m}
  \end{equation*}
  for all $\psi_m \in H^{1/2,1/4}(\Sigma_m)$. By Equation~\eqref{eq:first_part_interior_ibp_formula} the first term on the right-hand side of this equation coincides with the first term on the right-hand side of \eqref{eq:interior_ibp_formula}. Since 
  \begin{equation*}
    \vg_2 = -\left( \frac{\partial}{\partial t} (\vdtrace)'(\alpha \varphi \vn)\right) \ast G_\alpha
      = - \alpha \frac{\partial}{\partial t} \left( ( (\vdtrace)'(\varphi \vn) ) \ast G_\alpha \right)
      = - \alpha \frac{\partial}{\partial t} \widetilde{V}(\varphi \vn),
  \end{equation*} 
  where $\widetilde{V}$ from $\eqref{eq:def_slp}$ is applied componentwise again, the equality of the other terms follows immediately.
\end{proof}

\subsubsection*{The integration by parts formula \eqref{eq:general_ibp_formula} as limit case of \eqref{eq:interior_ibp_formula}}

Proposition~\ref{prop:interior_ibp_formula} provides us with an integration by parts formula \eqref{eq:interior_ibp_formula} on artificial boundaries~$\Sigma_m$ inside of $Q$. The final step in the proof of Theorem~\ref{thm:general_ibp_formula} is to deduce the actual integration by parts formula \eqref{eq:general_ibp_formula} therefrom. For this purpose, we consider a sequence $\{\Omega_m\}_m$ of smooth domains approximating $\Omega$ as established by Theorem~\ref{thm:approximating_domains} and denote $Q_m := \Omega_m \times (0,T)$ and $\Sigma_m := \Gamma_m \times (0,T)$ as before. We will refer to $\{\Sigma_m\}_m$ in the following as a \emph{smooth approximating sequence} of $\Sigma$. Let $\varphi, \psi \in H^{1/2,1/4}(\Sigma)$ and define $\psi_m \in H^{1/2,1/4}(\Sigma_m)$ to be the restriction of the extension $\rinvdtraceend \psi \in H^{1,1/2}_{;,0}(Q)$ of $\psi$ to~$\Sigma_m$, i.e.
\begin{equation} \label{eq:def_psi_m}
  \psi_m := \dtracem ( \rinvdtraceend \psi ).
\end{equation}
Then, the left-hand side and the first term on the right-hand side of \eqref{eq:interior_ibp_formula} converge to the respective terms of \eqref{eq:general_ibp_formula} in the limit as $m$ tends to infinity. This is the content of the next two lemmata. For the convergence of the remaining term, which is handled in Lemma~\ref{lem:convergence_time_bf}, additional assumptions on $\psi$ are required.

\begin{lemma} \label{lem:convergence_bf_ntrace}
  Let $\{\Sigma_m\}_m$ be a smooth approximating sequence of $\Sigma$ as introduced at the beginning of the paragraph.
  Let $\varphi, \psi \in H^{1/2,1/4}(\Sigma)$, $u = W\varphi$, and $\psi_m$ be defined by \eqref{eq:def_psi_m}. Then
  \begin{equation}
    \lim_{m\rightarrow\infty}  -\alpha\,\langle \ntracem u, \psi_m \rangle_{\Sigma_m} 
    = \langle D \varphi, \psi \rangle_\Sigma.
  \end{equation} 
\end{lemma}

\begin{proof}
  By the definition of $\ntrace$ via Green's first identity in \eqref{eq:def_ntrace} there holds
  \begin{equation*}
    \langle D \varphi, \psi \rangle_\Sigma = -\alpha \langle \ntrace u, \psi \rangle_\Sigma
      = -\alpha\left( \int_0^T \int_\Omega \nabla u \cdot \nabla (\rinvdtraceend \psi) \,\dif\vx\,\dif t 
        + d(u, \rinvdtraceend \psi) \right),
  \end{equation*}
  where we use that $(\partial/\partial t - \alpha \upDelta) u = 0$, since $u = W \varphi$. The Neumann trace on $\Sigma_m$ is defined analogously and thus 
  \begin{equation*}
    -\alpha\,\langle \ntracem u, \psi_m \rangle_{\Sigma_m} 
      = -\alpha\left( \int_0^T \int_{\Omega_m} \nabla u \cdot \nabla (\rinvdtraceend \psi) \,\dif\vx\,\dif t 
      + d_{Q_m}(u, \rinvdtraceend \psi) \right),
  \end{equation*}
  where we use that $(\rinvdtraceend \psi)|_{Q_m}$ is an extension of $\psi_m$ into $H^{1,1/2}_{;,0}(Q_m)$ and the bilinear form $d_{Q_m}: H^{1,1/2}_{;0,}(Q_m) \times H^{1,1/2}_{;,0}(Q_m) \rightarrow \mathbb{R}$ is defined in the same way as the bilinear form $d$ in \eqref{eq:def_bilinear_form_d} replacing only $\Omega$ by $\Omega_m$. Note that we identify here and in the following $u$ and $\rinvdtraceend \psi$ with the restrictions $u|_{Q_m}$ and $(\rinvdtraceend \psi)|_{Q_m}$, respectively, when operating on $Q_m$ to simplify the notation. By subtracting the second equation from  the first it follows that
  \begin{equation} \label{eq:estimate_difference_bf_ntrace}
    \begin{split}
      \Big| \langle D \varphi, \psi \rangle_\Sigma &+ \alpha\,\langle \ntracem u, \psi_m \rangle_{\Sigma_m}  \Big| \\
      &\leq \alpha \left| \int_0^T \int_{\Omega \backslash \Omega_m}  
        \nabla u \cdot \nabla (\rinvdtraceend \psi) \,\dif\vx\,\dif t \, \right| 
      + \alpha \left| d(u, \rinvdtraceend \psi) - d_{Q_m}(u, \rinvdtraceend \psi) \right|.
    \end{split}
  \end{equation}
  The first term on the right-hand side converges to 0 as $m \rightarrow \infty$, since $u$, $\rinvdtraceend \psi \in H^{1,1/2}(Q)$ implies $\nabla u$, $\nabla \rinvdtraceend \psi \in \mathbf{L}^2(Q)$ and $|\Omega \backslash \Omega_m| \rightarrow 0$.

  It is slightly more difficult to see that the second term in \eqref{eq:estimate_difference_bf_ntrace} converges to zero. The problem is that $d$ and $d_{Q_m}$ are defined only as continuous extensions of \eqref{eq:def_bilinear_form_d} for general functions in $H^{1,1/2}_{;0,}(Q)$ and $H^{1,1/2}_{;,0}(Q)$. Therefore, let $\{u_n\}_n$ be a sequence of functions in $C^\infty_c(\overline{\Omega}\times(0,T])$  converging to $u$ in $H^{1,1/2}_{;0,}(Q)$ and $\{v_n\}_n$ be a sequence in $C^\infty_c(\overline{\Omega}\times[0,T))$ converging to $v := \rinvdtraceend \psi$ in~$H^{1,1/2}_{;,0}(Q)$. The restrictions $u_n|_{Q_m} \in C^\infty_c(\overline{\Omega}_m\times(0,T])$ and $v_n|_{\Omega_m} \in C^\infty_c(\overline{\Omega}_m\times[0,T))$ of these functions converge to $u|_{Q_m}$ in $H^{1,1/2}_{;0,}(Q_m)$ and $v|_{Q_m}$ in $H^{1,1/2}_{;,0}(Q_m)$, respectively. In particular, it follows that
  \begin{equation*}
    \lim_{n\rightarrow\infty} d(u_n,v_n) = d(u, v), \qquad
    \lim_{n\rightarrow\infty} d_{Q_m}(u_n,v_n) = d_{Q_m}(u, v),
  \end{equation*}
  by the continuity of $d_Q$ and $d_{Q_m}$ stated in Proposition~\ref{prop:bilinear_form_d}. This motivates us to estimate
  \begin{equation} \label{eq:estimate_difference_d}
    \begin{split}
      | d(u,v) - d_{Q_m}(u,v) | \leq\; &| d(u,v) - d(u_n,v_n) |+ | d(u_n,v_n) - d_{Q_m}(u_n,v_n) | \\
        &+ |  d_{Q_m}(u_n,v_n) - d_{Q_m}(u,v) |
    \end{split}
  \end{equation}
  and to show that the right-hand side can be bounded by an arbitrarily small $\varepsilon$ for a suitably chosen~$n$ and sufficiently large $m$.
  
  Let $\varepsilon > 0$ be fixed. The last summand in \eqref{eq:estimate_difference_d} can be estimated by 
  \begin{align*}
    |  d_{Q_m}(&u_n,v_n) - d_{Q_m}(u,v) | 
      \leq |d_{Q_m}(u_n,v_n) - d_{Q_m}(u,v_n)| + |d_{Q_m}(u,v_n) - d_{Q_m}(u,v)| \\
      &\leq c_d \|u-u_n\|_{H^{1,1/2}_{;0,}(Q_m)} \|v_n\|_{H^{1,1/2}_{;,0}(Q_m)}
        + c_d \|u\|_{H^{1,1/2}_{;0,}(Q_m)} \|v - v_n\|_{H^{1,1/2}_{;,0}(Q_m)} \\
      &\leq c_d \left( \sup_{n \in \mathbb{N}}\Big\{\|v_n\|_{H^{1,1/2}_{;,0}(Q)}\Big\} + \|u\|_{H^{1,1/2}_{;0,}(Q)} \right) 
        \left( \|u-u_n\|_{H^{1,1/2}_{;0,}(Q)} + \|v - v_n\|_{H^{1,1/2}_{;,0}(Q)} \right).
  \end{align*}
  Here we used that the Sobolev norms of functions restricted to $Q_m$ can be estimated by the respective norms on $Q$ and that the bilinear forms $d_{Q_m}$ can be bounded by a constant $c_d$ independent of the domains $\Omega_m$, as stated in Proposition~\ref{prop:bilinear_form_d}. Due to the convergence of $u_n$ to~$u$ and $v_n$ to~$v$ there exists an~$n(\varepsilon)$ independent of $m$ such that for all $n > n(\varepsilon)$ 
  \begin{equation*}
    \max\Big\{\|u-u_n\|_{H^{1,1/2}_{;0,}(Q)}, \|v - v_n\|_{H^{1,1/2}_{;,0}(Q)} \Big\} < \frac{\varepsilon}
      {6 c_d} \left( \sup_{n \in \mathbb{N}}\Big\{\|v_n\|_{H^{1,1/2}_{;,0}(Q)}\Big\} + \|u\|_{H^{1,1/2}_{;0,}(Q)} \right)^{-1}.
  \end{equation*}
  For all such $n$ we conclude that 
  \begin{equation*}
    |  d_{Q_m}(u_n,v_n) - d_{Q_m}(u,v) | < \frac{\varepsilon}{3}.
  \end{equation*}
  By repeating the same arguments for the first term in \eqref{eq:estimate_difference_d} it follows that
  \begin{equation*}
    | d(u,v) - d(u_n,v_n) | < \frac{\varepsilon}{3}
  \end{equation*} 
  for all $n > n(\varepsilon)$, where $n(\varepsilon)$ can be chosen to be the same for both terms.

  The second summand in \eqref{eq:estimate_difference_d} is given by
  \begin{align*}
    | d(u_n,v_n) - d_{Q_m}(u_n,v_n) | &= 
      \left| \int_0^T\int_\Omega \frac{\partial u_n}{\partial t}(\vx,t) v_n(\vx,t) \,\dif\vx\,\dif t
      - \int_0^T\int_{\Omega_m} \frac{\partial u_n}{\partial t}(\vx,t) v_n(\vx,t) \,\dif\vx\,\dif t \right| \\
      &=  \left| \int_0^T\int_{\Omega \backslash \Omega_m} \frac{\partial u_n}{\partial t}(\vx,t) v_n(\vx,t) \,\dif\vx\,\dif t \right|.
  \end{align*}
  Note that we can use the explicit representation \eqref{eq:def_bilinear_form_d} of $d$ and $d_{Q_m}$ since $u_n$ and $v_n$ are smooth. The convergence of $|\Omega \backslash \Omega_m|$ to zero as $m$ tends to infinity allows us to find an $m(\varepsilon,n)$ such that 
  \begin{equation*}
    | d(u_n,v_n) - d_{Q_m}(u_n,v_n) | < \frac{\varepsilon}{3}
  \end{equation*} 
  for all $m>m(\varepsilon,n)$. Hence, for a fixed $n > n(\varepsilon)$ and all $m > m(\varepsilon,n)$ we can bound the 
  right-hand side of \eqref{eq:estimate_difference_d} by $\varepsilon$. Therefore, both terms on the right-hand side of \eqref{eq:estimate_difference_bf_ntrace} converge to zero as~$m \rightarrow \infty$.
\end{proof}

\begin{lemma}
  Let $\{\Sigma_m\}_m$ be a smooth approximating sequence of $\Sigma$ as introduced at the beginning of the paragraph.
  Let $\varphi, \psi \in H^{1/2,1/4}(\Sigma)$ and $\psi_m$ be defined by \eqref{eq:def_psi_m}. Then 
  \begin{equation}
    \lim_{m\rightarrow\infty} \langle \scurlm \psi_m, \dtracem \widetilde{V}(\curl_{\Sigma} \varphi) \rangle_{\Sigma_m} 
    = \langle \scurl \psi, V (\scurl \varphi) \rangle_{\Sigma}.
  \end{equation} 
\end{lemma}

\begin{proof}
  By the definition of the surface curl in~\eqref{eq:def_scurl} there holds
  \begin{align*}
    \langle \scurl \psi, V (\scurl \varphi) \rangle_{\Sigma} 
    &= \langle \nabla \rinvdtraceend \psi, \curl( \widetilde{V} (\scurl \varphi) ) \rangle_{\mathbf{L}^2(Q)} \\
    &= \int_0^T \int_\Omega (\nabla \rinvdtraceend \psi)(\vx,t) \cdot \curl( \widetilde{V} (\scurl \varphi) )(\vx,t)
      \,\dif\vx\,\dif t,
  \end{align*}
  where we used that $\widetilde{V} (\scurl \varphi)$ is an extension of $V (\scurl \varphi)$ and that the definition in~\eqref{eq:def_scurl} is independent of the extensions of the the function $\psi$ and the test function, see Proposition~\ref{prop:surface_curl}. Similarly it follows that 
  \begin{equation*}
    \langle \scurlm \psi_m, \dtracem \widetilde{V}(\curl_{\Sigma} \varphi) \rangle_{\Sigma_m} 
    = \int_0^T \int_{\Omega_m} (\nabla \rinvdtraceend \psi)(\vx,t) \cdot 
      \curl( \widetilde{V} (\scurl \varphi) )(\vx,t)
      \,\dif\vx\,\dif t,
  \end{equation*} 
  by using in addition, that $(\rinvdtraceend \psi)|_{Q_m}$ is an extension of $\psi_m$ due to its definition in \eqref{eq:def_psi_m}. Therefore
  \begin{align*}
    \Big| \langle \scurl \psi, V (\scurl \varphi) \rangle_{\Sigma} 
      &- \langle \scurlm \psi_m, \dtracem \widetilde{V}(\curl_{\Sigma} \varphi) \rangle_{\Sigma_m} \Big| \\
    &= \left| \int_0^T \int_{\Omega \backslash \Omega_m} 
        (\nabla \rinvdtraceend \psi)(\vx,t) \cdot \curl( \widetilde{V} (\scurl \varphi) )(\vx,t)
        \,\dif\vx\,\dif t \right|.
  \end{align*}
  Note that $(\nabla \rinvdtraceend \psi) \in \mathbf{L}^2(Q)$ and so is $\curl( \widetilde{V} (\scurl \varphi) )$, due to $\widetilde{V} (\scurl \varphi) \in \mathbf{H}^{1,1/2}_{;0,}(Q)$. Hence, the right-hand side in the last equation converges to zero as $m \rightarrow \infty$ because ${|\Omega\backslash\Omega_m| \rightarrow 0}$.
\end{proof}

To show a similar convergence result for the second term of the right-hand side in \eqref{eq:interior_ibp_formula} we need to require the test function $\psi$ to be more regular. The following result holds.

\begin{lemma} \label{lem:convergence_time_bf}
  Let $\varphi \in H^{1/2,1/4}(\Sigma)$, $\psi \in \dtrace(C^\infty_c(\overline{\Omega}\times(0,T)))$ and $\widetilde{\psi} \in C^\infty_c(\mathbb{R}^3\times(0,T))$ be such that $\psi = \widetilde{\psi}|_{\Sigma}$. Let~$\{\Sigma_m\}_m$ be a smooth approximating sequence of $\Sigma$ as introduced at the beginning of the paragraph and $\psi_m := \widetilde{\psi}|_{\Sigma_m}$. Then 
  \begin{equation} \label{eq:convergence_time_bf}
    \lim_{m\rightarrow\infty} \left\langle \nvecm\cdot 
      \dtracem \left(\frac{\partial}{\partial t} \widetilde{V} (\varphi \vn ) \right), 
      \psi_m \right\rangle_{\Sigma_m} 
    = b( \varphi, \psi )
  \end{equation}
  with $b$ defined in \eqref{eq:def_b}.
\end{lemma}

\begin{proof}
  We start by showing the identity
  \begin{equation} \label{eq:distributional_identity_time_bf}
    \left\langle \nvecm\cdot \dtracem \left(\frac{\partial}{\partial t} \widetilde{V} (\varphi \vn ) \right), 
      \psi_m \right\rangle_{\Sigma_m} 
    = \frac{\partial}{\partial t} (\dtracem)'
      \left( \dtracem (\widetilde{V}(\varphi \vn)) \cdot \nvecm\right)[ \widetilde{\psi} ],
  \end{equation}
  where the right-hand side is understood as application of a distribution in $\mathcal{D}'(\mathbb{R}^3 \times (0,T))$ to $\widetilde{\psi} \in C^\infty_c(\mathbb{R}^3 \times (0,T))$. For the duality product on the left-hand side there holds
  \begin{align*}
    \left\langle \nvecm\cdot \dtracem \left(\frac{\partial}{\partial t} \widetilde{V} (\varphi \vn ) \right), 
      \psi_m \right\rangle_{\Sigma_m} 
    &= \int_0^T\int_{\Gamma_m} 
      \frac{\partial}{\partial t} (\widetilde{V} (\varphi \vn ))(\vx,t) \cdot \nvecm(\vx) 
      \widetilde{\psi}(\vx,t) \,\dif\vs_\vx\,\dif t \\
    &= -\int_0^T\int_{\Gamma_m} 
      \widetilde{V} (\varphi \vn)(\vx,t) \cdot \nvecm(\vx) 
      \frac{\partial \widetilde{\psi}}{\partial t}(\vx,t) \,\dif\vs_\vx\,\dif t,
  \end{align*}
  where we used classical integration by parts in the second step, which is possible because ${\widetilde{V} (\varphi \vn ) \in C^\infty(Q)}$, $\widetilde{\psi} \in C^\infty_c(\mathbb{R}^3 \times (0,T))$ and $\Sigma_m \subset Q$. The right-hand side of this equation can be interpreted as a duality product on $\mathcal{D}'(\mathbb{R}^3 \times (0,T)) \times C^\infty_c(\mathbb{R}^3 \times (0,T))$ by
  \begin{equation*}
    - (\dtracem)'\left( \dtracem (\widetilde{V}(\varphi \vn))\cdot\nvecm\right) 
        \left[ \frac{\partial \widetilde{\psi}}{\partial t} \right] 
    = \frac{\partial}{\partial t} (\dtracem)'\left( \dtracem (\widetilde{V}(\varphi \vn))
      \cdot\nvecm\right) [ \widetilde{\psi} ],
  \end{equation*}
  where the last equality is just the definition of the distributional time derivative. Therefore, \eqref{eq:distributional_identity_time_bf} holds true.

  Due to \eqref{eq:distributional_identity_time_bf} the convergence in \eqref{eq:convergence_time_bf} follows if we can show that
  \begin{equation} \label{eq:distributional_convergence_time_bf}
    \frac{\partial}{\partial t} (\dtracem)'
      \left( \dtracem (\widetilde{V}(\varphi \vn)) \cdot \nvecm\right)
    \rightarrow \frac{\partial}{\partial t} (\dtrace)' \Big( V(\varphi \vn) \cdot \vn \Big)
      \quad \text{in } \mathcal{D}'(\mathbb{R}^3 \times (0,T)).
  \end{equation}
  For this purpose, let $w \in C^\infty_c(\mathbb{R}^3\times(0,T))$. Then 
  \begin{align*}
    \bigg| & \frac{\partial}{\partial t} (\dtrace)' \Big( V(\varphi \vn) \cdot \vn \Big)[w]
      - \frac{\partial}{\partial t} 
      (\dtracem)' \Big( \dtracem (\widetilde{V}(\varphi \vn)) \cdot \nvecm \Big)[w] \bigg| \\
    &=\bigg| \int_0^T\int_{\Gamma} V(\varphi \vn)(\vx,t)\cdot \vn(\vx) \frac{\partial w}{\partial t}(\vx,t)
      \,\dif\vs_\vx\,\dif t
      - \int_0^T\int_{\Gamma_m} \widetilde{V}(\varphi \vn)(\vx,t)\cdot \nvecm(\vx) \frac{\partial w}{\partial t}(\vx,t)
      \,\dif\vs_\vx\,\dif t \bigg|
  \end{align*}
  Since the product of $\partial w/\partial t$ and $\widetilde{V}(\varphi \vn)$ is in $L^2(0,T;\mathbf{H}^1(\Omega))$ we can apply the divergence theorem to both surface integrals in the previous line and get
  \begin{align*}
    \bigg|
      \int_0^T \int_\Omega \bigg( \diver( \widetilde{V}(\varphi \vn))(\vx,t) \frac{\partial w}{\partial t}(\vx,t)
          &+ \widetilde{V}(\varphi \vn)(\vx,t) \cdot \nabla \left( \frac{\partial w}{\partial t} \right)(\vx,t)
          \bigg) \,\dif \vx\,\dif t \\
      - \int_0^T \int_{\Omega_m} \bigg( \diver( \widetilde{V}(\varphi \vn))&(\vx,t) \frac{\partial w}{\partial t}(\vx,t)
          + \widetilde{V}(\varphi \vn)(\vx,t) \cdot \nabla \left( \frac{\partial w}{\partial t} \right)(\vx,t)
          \bigg) \,\dif \vx\,\dif t
    \bigg| \\
    \leq \bigg| \int_0^T\int_{\Omega\backslash\Omega_m} \diver( \widetilde{V}(\varphi \vn))(\vx,t) 
        \frac{\partial w}{\partial t}(\vx,t) \,&\dif \vx\,\dif t \bigg|
      + \bigg| \int_0^T\int_{\Omega\backslash\Omega_m} \widetilde{V}(\varphi \vn)(\vx,t) \cdot 
        \nabla \left( \frac{\partial w}{\partial t} \right)(\vx,t) \,\dif \vx\,\dif t \bigg| \\
    \leq \|\diver( \widetilde{V}(\varphi \vn)) \|_{L^2(Q\backslash Q_m)}
        \bigg\| \frac{\partial w}{\partial t} \bigg\|_{L^2(Q\backslash Q_m)}
      &+ \| \widetilde{V}(\varphi \vn) \|_{\mathbf{L}^2(Q\backslash Q_m)}
        \bigg\| \nabla \left( \frac{\partial w}{\partial t} \right) \bigg\|_{\mathbf{L}^2(Q\backslash Q_m)}.
  \end{align*}
  The norms in the last line are bounded for all $m$ and converge to 0 as $m\rightarrow\infty$, since~${|Q\backslash Q_m| \rightarrow 0}$. Therefore, we have established the convergence in \eqref{eq:distributional_convergence_time_bf} and thus in~\eqref{eq:convergence_time_bf}.
\end{proof}

Finally, we are ready to give a proof of Theorem~\ref{thm:general_ibp_formula} by collecting all the results.
\begin{proof}[\textbf{Proof of Theorem~\ref{thm:general_ibp_formula}}]
  Let~$\{\Sigma_m\}_m$ be a smooth approximating sequence of $\Sigma$ as introduced at the beginning of the paragraph. As in Lemma~\ref{lem:convergence_time_bf} we assume first that the test function $\psi \in \dtrace(C^\infty_c(\overline{\Omega} \times (0,T)))$, i.e.~$\psi = \widetilde{\psi}|_{\Sigma}$ for some $\widetilde{\psi} \in C^\infty_c(\mathbb{R}^3 \times (0,T))$, and denote $\psi_m := \widetilde{\psi}|_{\Sigma_m}$.
  On each boundary $\Sigma_m$ the integration by parts formula \eqref{eq:interior_ibp_formula} derived in Proposition~\ref{prop:interior_ibp_formula} is valid for $\varphi$ and $\psi_m$. By applying Lemmata \ref{lem:convergence_bf_ntrace}--\ref{lem:convergence_time_bf} the limit of \eqref{eq:interior_ibp_formula} as $m\rightarrow \infty$ is given by
  \begin{equation*}
    \langle D \varphi, \psi \rangle_\Sigma = 
      \alpha^2 \langle \scurl \psi, V (\scurl \varphi) \rangle_{\Sigma} + \alpha \, b(\varphi, \psi),
  \end{equation*}
  which is the desired integration by parts formula \eqref{eq:general_ibp_formula} on $\Sigma$. 

  It remains to show that \eqref{eq:general_ibp_formula} holds also for general $\psi \in H^{1/2,1/4}(\Sigma)$ in the appropriate sense. By reordering the terms in the integration by parts formula we get 
  \begin{equation*}
    \alpha \, b(\varphi, \psi) = 
      \langle D \varphi, \psi \rangle_\Sigma - \alpha^2 \langle \scurl \psi, V (\scurl \varphi) \rangle_{\Sigma}
  \end{equation*} 
  for $\psi \in \dtrace(C^\infty_c(\overline{\Omega}\times(0,T)))$. Since the three operators $D: H^{1/2,1/4}(\Sigma) \rightarrow H^{-1/2,-1/4}(\Sigma)$, ${\scurl: H^{1/2,1/4}(\Sigma) \rightarrow \mathbf{H}^{-1/2,-1/4}(\Sigma)}$ and $V: \mathbf{H}^{-1/2,-1/4}(\Sigma) \rightarrow \mathbf{H}^{1/2,1/4}(\Sigma)$ are continuous, the right-hand side of that equation interpreted as a bilinear form on ${H^{1/2,1/4}(\Sigma) \times H^{1/2,1/4}(\Sigma)}$ is continuous too. Hence, also the left-hand side, i.e.~the bilinear form 
  \begin{align*}
    b(\cdot,\cdot):\ &H^{1/2,1/4}(\Sigma) \times \dtrace(C^\infty_c(\overline{\Omega}\times(0,T))) 
      \rightarrow \mathbb{R} \\
    &(\varphi,\psi) \mapsto b(\varphi,\psi) 
      = \left\langle V (\varphi \vn), \frac{\partial}{\partial t} \psi \vn \right\rangle_{\Sigma}
  \end{align*}
  is continuous with respect to the norm in $H^{1/2,1/4}(\Sigma) \times H^{1/2,1/4}(\Sigma)$ and it admits a unique, continuous extension to this space due to the density of $\dtrace(C^\infty_c(\overline{\Omega}\times(0,T)))$ in $H^{1/2,1/4}(\Sigma)$. In particular, \eqref{eq:general_ibp_formula} holds for general $\varphi,\psi \in H^{1/2,1/4}(\Sigma)$ by continuity, if we identify $b$ with its continuous extension.
\end{proof}

\section{The bilinear form \texorpdfstring{$b(\cdot,\cdot)$}{b}} \label{sec:bf_time_derivative}
In this section we focus on special situations in which we can express the bilinear form $b(\cdot, \cdot)$ on the right-hand side of the general integration by parts formula \eqref{eq:general_ibp_formula} in terms of weakly singular integrals.

To start off, we consider representation \eqref{eq:formal_def_second_bilinear_form} and show that the corresponding integral kernel $\partial G_\alpha/ \partial\tau$ is not Lebesgue integrable on $\Sigma \times \Sigma$.

\begin{proposition} \label{prop:kernel_not_integrable}
  The function $((\vx,t),(\vy,\tau)) \mapsto \partial G_\alpha/ \partial\tau (\vx-\vy,t-\tau)$ is not Lebesgue integrable on $\Sigma \times \Sigma$.
\end{proposition}

The proposition can be shown by a direct computation for arbitrary Lipschitz domains. For smooth boundaries $\Gamma$ a simpler proof is available, which we present in the following.
\begin{proof}
  Let $\Gamma$ be smooth. For $\varepsilon > 0$ we consider the integral 
  \begin{equation} \label{eq:def_I_eps}
    I_\varepsilon := \int_0^T \int_\Gamma \int_0^{(t-\varepsilon)^+} \int_\Gamma
      \frac{\partial G_\alpha}{\partial \tau} (\vx-\vy,t-\tau)\, \dif\vs_\vy\, \dif\tau\, \dif\vs_\vx\, \dif t,
  \end{equation}
  where $(t-\varepsilon)^+:=\max(t-\varepsilon, 0)$. It suffices to show that $I_\varepsilon$ diverges as $\varepsilon \rightarrow 0$. Since the kernel $((\vx,t),(\vy,\tau)) \mapsto \partial G_\alpha/ \partial\tau (\vx-\vy,t-\tau)$ is smooth on the integration domain in \eqref{eq:def_I_eps} for all $\varepsilon > 0$, we can first integrate with respect to the time variables and get
  \begin{align*}
    \int_0^T &\int_0^{(t-\varepsilon)^+} 
    \frac{\partial G_\alpha}{\partial \tau} (\vx-\vy,t-\tau)\, \dif\tau\, \dif t = 
    \int_\varepsilon^T \left( G_\alpha(\vx-\vy, \varepsilon ) - G_\alpha(\vx-\vy, t ) \right) \dif t\\ &= 
    \frac{T - \varepsilon}{(4\pi\alpha \varepsilon)^{3/2}}\exp\bigg( -\frac{|\vx-\vy|^2}{4\alpha \varepsilon} \bigg) 
    + \frac{1}{4 \pi \alpha |\vx-\vy|} \erf\left( \frac{|x-y|}{2\sqrt{\alpha T}}\right)
    - \frac{1}{4 \pi \alpha |\vx-\vy|} \erf\left( \frac{|x-y|}{2\sqrt{\alpha \varepsilon}}\right).
  \end{align*}
  The integrals of the last two terms over $\Gamma \times \Gamma$ are uniformly bounded for arbitrary $\varepsilon > 0$ since the integrands are products of an error function, which is bounded on $\mathbb{R}$, and a scaled Laplace kernel, which is integrable on $\Gamma \times \Gamma$.

  For the first term one can show that \cite[cf.~Theorem~2.1 for smooth $\Gamma$]{Tau2009}
  \begin{equation} \label{eq:divergent_part_I_eps}
    \int_\Gamma \frac{T - \varepsilon}{(4\pi\alpha \varepsilon)^{3/2}}\exp\bigg( -\frac{|\vx-\vy|^2}{4\alpha \varepsilon} \bigg) \dif\vs_\vy
      = \frac{T - \varepsilon}{\sqrt{4 \pi \alpha \varepsilon}} ( 1 + \varepsilon\, \widetilde{v}_\alpha(\vx,\varepsilon)) 
  \end{equation}
  where $(\vx, t) \mapsto \widetilde{v}_\alpha(\vx,t)$ is bounded on $\Gamma \times [0,T]$. The integral of \eqref{eq:divergent_part_I_eps} with respect to $\vx$ over $\Gamma$ is~$\mathcal{O}(\varepsilon^{-1/2})$. Hence, $I_\varepsilon = \mathcal{O}(\varepsilon^{-1/2})$ and thus it diverges as $\varepsilon \rightarrow 0$.
\end{proof}

Proposition~\ref{prop:kernel_not_integrable} indicates that \eqref{eq:formal_def_second_bilinear_form} is not an appropriate representation of the bilinear form~$b$. Nonetheless, it was used in applications where it yielded satisfactory results, see e.g.~\cite[Example~3.2]{Mes2014}. Let us further comment on this.

The standard way to discretize the bilinear forms induced by boundary integral operators of the heat equation is to use a tensor product approach. To start off, the space-time boundary~$\Sigma$ is approximated by a tensor product decomposition $\Sigma_h = I_{N_t} \times \Gamma_h$, where $I_{N_t} := \{\tau_k\}_{k=1}^{N_t}$ is a decomposition of the time interval $(0,T)$ into $N_t$ pairwise disjoint, possibly non-uniform, open intervals $\tau_k = (t_{k-1},t_k)$ and $\Gamma_h$ is an admissible triangular mesh approximating the spatial boundary~$\Gamma$. Let $S_{h_t}^{0}(I_{N_t})$ be the space of piecewise constant functions defined on $I_{N_t}$ and $S_{h_x}^1(\Gamma_h)$ be the space of piecewise linear and globally continuous functions on the mesh~$\Gamma_h$. Then the tensor product space
\begin{equation} \label{eq:lowest_order_ansatz_functions}
  S_{h_x,h_t}^{1 \otimes 0}(\Sigma_h) := S_{h_x}^1(\Gamma_h) \otimes S_{h_t}^{0}(I_{N_t})
\end{equation}
is a subspace of $H^{1/2,1/4}(\Sigma)$ typically used for its discretization. For $\varphi_{\vx} \in S_{h_x}^1(\Gamma_h)$ and the indicator function $\mathbbm{1}_{\tau_j}$ of an interval $\tau_j \in I_{N_t}$ we have $\mathbbm{1}_{\tau_j} \otimes \varphi_{\vx} \in S_{h_x,h_t}^{1 \otimes 0}(\Sigma_h)$. By a very formal computation we can evaluate ${\alpha\, b( \mathbbm{1}_{\tau_j} \otimes \varphi_{\vx}, \mathbbm{1}_{\tau_j} \otimes \varphi_{\vx})}$ via \eqref{eq:formal_def_second_bilinear_form} yielding 
\begin{align} 
  \alpha\, b(\mathbbm{1}_{\tau} &\otimes\varphi_{\vx}, \mathbbm{1}_{\tau} \otimes \varphi_{\vx} ) \notag \\
  &=- \alpha \int_{t_{j-1}}^{t_j} \int_{\Gamma} \varphi_{\vx}(\vx) \vn(\vx)\cdot \int_{t_{j-1}}^t \int_{\Gamma} 
    \frac{\partial G_{\alpha}}{\partial\tau}(\vx-\vy, t-\tau) \varphi_{\vx}(\vy) \vn(\vy) 
    \, \dif\vs_\vy \, \dif\tau \, \dif\vs_\vx \, \dif t \notag \\
  &= - \alpha \int_{\Gamma} \int_{\Gamma} \varphi_{\vx}(\vx) \varphi_{\vx}(\vy) \vn(\vx)\cdot \vn(\vy) 
    \int_{t_{j-1}}^{t_j} \int_{t_{j-1}}^t \frac{\partial G_{\alpha}}{\partial\tau}(\vx-\vy, t-\tau) 
    \, \dif\tau \, \dif t \, \dif\vs_\vy \, \dif\vs_\vx \notag  \\
  &= \alpha \int_{\Gamma} \int_{\Gamma} \varphi_{\vx}(\vx) \varphi_{\vx}(\vy) \vn(\vx)\cdot \vn(\vy) 
    \int_{t_{j-1}}^{t_j} G_\alpha(\vx-\vy, t-t_{j-1}) \, \dif t \, \dif\vs_\vy \, \dif\vs_\vx, \label{eq:formal_primitive}
\end{align}
where we used that $\partial G_\alpha/ \partial\tau (\vx - \vy, 0)$ is zero if $\vx-\vy \neq \vzero$. The remaining expression is then unproblematic since the kernel $((\vx,t),\vy) \mapsto G_\alpha(\vx-\vy, t-t_{j-1})$ is Lebesgue integrable on~$(\Gamma \times \tau_j) \times \Gamma$. However, the steps taken to get this expression cannot be easily justified. Indeed, one cannot apply the classical Fubini theorem to change the order of integration, since the kernel $((\vx,t),(\vy,\tau)) \mapsto \partial G_\alpha/ \partial\tau(\vx-\vy,t-\tau)$ is not Lebesgue integrable on the integration domain as seen in Proposition~\ref{prop:kernel_not_integrable}, and one cannot ignore the singularities of $(\vx, \vy) \mapsto \partial G_\alpha/ \partial\tau (\vx - \vy, 0)$ at~$\vx=\vy$. Nonetheless, numerical results indicate that the result in \eqref{eq:formal_primitive} holds true. This motivates us to find a representation of $\alpha\, b(\cdot, \cdot)$ that is similar to \eqref{eq:formal_def_second_bilinear_form} but overcomes the problem of the locally non-integrable kernel.

\subsection{An integral representation of the bilinear form \texorpdfstring{$b$}{b} in a tensor product setting} \label{subsec:integral_representation_b}
Let $I_{N_t} := \{\tau_k\}_{k=1}^{N_t}$ be a partition of the time interval $(0,T)$ into $N_t$ pairwise disjoint, possibly non-uniform, open intervals $\tau_k = (t_{k-1},t_k)$. Define the space 
\begin{equation} \label{eq:def_c1_pw}
  C^1_\mathrm{pw}(I_{N_t}) := \{ \varphi \in L^\infty(0,T) : \varphi|_{\tau_k} \in C^1(\tau_k), (\varphi|_{\tau_k})' \in L^\infty(\tau_k) \}.
\end{equation}
The tensor product space $(L^\infty(\Gamma) \cap H^{1/2}(\Gamma)) \otimes C^1_\mathrm{pw}(I_{N_t})$ is a subspace of~$H^{1/2,1/4}(\Sigma)$. For functions in this subspace the bilinear form $b$ admits the following representation.
\begin{theorem} \label{thm:second_bilinear_form_tensor_case}
  Let $\varphi \in (L^\infty(\Gamma) \cap H^{1/2}(\Gamma)) \otimes C^1_\mathrm{pw}(I_{N_t})$ and $\psi \in H^{1/2,1/4}(\Sigma) \cap L^\infty(\Sigma)$. Then
  \begin{align} \label{eq:second_bilinear_form_tensor_case}
    b(\varphi,\psi) = \sum_{k=1}^{N_t} \int_{t_{k-1}}^{t_k} \int_\Gamma \psi(\vx,t) \vn(\vx)\cdot \Bigg[ 
      &-\int_0^{t_{k-1}} \int_\Gamma \frac{\partial G_\alpha}{\partial\tau}(\vx-\vy,t-\tau) \varphi(\vy,\tau)\vn(\vy)
        \,\dif\vs_\vy\,\dif\tau \notag \\
      &+ \int_\Gamma G_\alpha(\vx-\vy,t-t_{k-1}) \varphi(\vy,t_{k-1}{+})\vn(\vy)\,\dif\vs_\vy \\
      &+ \int_{t_{k-1}}^{t} \int_\Gamma 
        G_\alpha(\vx-\vy,t-\tau)\frac{\partial\varphi}{\partial\tau}(\vy,\tau) \vn(\vy)
        \,\dif\vs_\vy\,\dif\tau \Bigg] \,\dif\vs_\vx\,\dif t, \notag
  \end{align}
  where $\varphi(\cdot,t_{k-1}{+})$ denotes the right limit of $\varphi$ with respect to time in $t_{k-1}$. In particular, all occurring integrands are Lebesgue integrable on the respective integration domains.
\end{theorem}
The representation \eqref{eq:second_bilinear_form_tensor_case} is tailored to boundary element methods with tensor product spaces used for discretization. The assumptions on $\varphi$ and $\psi$ are satisfied in this context. For example, let $\Gamma_h$ be a triangulation of a polyhedral boundary and $\Sigma_h = \Gamma_h \times I_{N_t}$. Then $\varphi, \psi \in S_{h_x,h_t}^{1 \otimes 0}(\Sigma_h)$ satisfy the assumptions of Theorem~\ref{thm:second_bilinear_form_tensor_case} and thus
\begin{equation}
  \begin{split}
    b(\varphi,\psi) = \sum_{k=1}^{N_t} \int_{t_{k-1}}^{t_k} \int_{\Gamma_h} \psi(\vx,t) \vn(\vx)\cdot \Bigg[ 
      &-\int_0^{t_{k-1}} \int_{\Gamma_h} \frac{\partial G_\alpha}{\partial\tau}(\vx-\vy,t-\tau) \varphi(\vy,\tau)\vn(\vy)
        \,\dif\vs_\vy\,\dif\tau \\
      &+ \int_{\Gamma_h} G_\alpha(\vx-\vy,t-t_{k-1}) \varphi(\vy,t_{k-1}{+})\vn(\vy)\,\dif\vs_\vy \Bigg] \,\dif\vs_\vx\,\dif t.
  \end{split}
\end{equation}
The second term in this representation is similar to what we have formally derived in \eqref{eq:formal_primitive}, so Theorem~\ref{thm:second_bilinear_form_tensor_case} justifies the calculations found in the literature. In \cite{ZapWatOfMer2021} it is described in detail how to deal with the remaining integrals.

We conclude this section with a proof of Theorem~\ref{thm:second_bilinear_form_tensor_case}. The strategy is similar to the one in the proof of Theorem~\ref{thm:general_ibp_formula}. We show a corresponding result on a sequence of auxiliary boundaries~$\{\Sigma_m\}_m$ first in Lemma \ref{lemma:second_bilinear_form_aux_boundary}, and then prove convergence when taking the limit with respect to $m$. 
\begin{lemma} \label{lemma:second_bilinear_form_aux_boundary}
  Let $\Omega_m$ be a Lipschitz domain with boundary $\Gamma_m$ such that $\overline{\Omega}_m \subset \Omega$. Let ${\Sigma_m = \Gamma_m \times (0,T)}$, $\varphi \in (L^\infty(\Gamma) \cap H^{1/2}(\Gamma)) \otimes C^1_\mathrm{pw}(I_{N_t})$ and $\psi_m \in H^{1/2,1/4}(\Sigma_m)$. Then
  \begin{equation} \label{eq:second_bilinear_form_aux_boundary}
    \begin{split}
      \bigg\langle &\nvecm\cdot \dtracem \left(\frac{\partial}{\partial t} \widetilde{V}(\varphi \vn ) \right), \psi_m\bigg\rangle_{\Sigma_m}\\ 
      &= \sum_{k=1}^{N_t} \int_{t_{k-1}}^{t_k} \int_{\Gamma_m} \psi_m(\vx,t) \vn_m(\vx)\cdot \Bigg[ 
      -\int_0^{t_{k-1}} \int_\Gamma \frac{\partial G_\alpha}{\partial\tau}(\vx-\vy,t-\tau) \varphi(\vy,\tau)\vn(\vy)
        \,\dif\vs_\vy\,\dif\tau \\
      &\qquad \qquad + \int_\Gamma G_\alpha(\vx-\vy,t-t_{k-1}) \varphi(\vy,t_{k-1}{+})\vn(\vy)\,\dif\vs_\vy \\
      &\qquad \qquad + \int_{t_{k-1}}^{t} \int_\Gamma 
        G_\alpha(\vx-\vy,t-\tau)\frac{\partial\varphi}{\partial\tau}(\vy,\tau) \vn(\vy)
        \,\dif\vs_\vy\,\dif\tau \Bigg] \,\dif\vs_\vx\,\dif t.
    \end{split}
  \end{equation}
\end{lemma}
\begin{proof}
  For $\vx \in \Gamma_m \subset \Omega$ there holds 
  \begin{equation} \label{eq:time_derivative_slp}
    \begin{split}
      \frac{\partial}{\partial t} \widetilde{V} (\varphi \vn ) (\vx, t) 
        &= \frac{\partial}{\partial t} \left( \int_0^t \int_\Gamma 
          G_\alpha( \vx-\vy,t-\tau) \varphi(\vy,\tau) \vn(\vy) \,\dif\vs_\vy\,\dif\tau \right) \\
        &= \int_0^t \int_\Gamma \frac{\partial G_\alpha}{\partial t}( \vx-\vy,t-\tau) 
          \varphi(\vy,\tau) \vn(\vy) \,\dif\vs_\vy\,\dif\tau, 
    \end{split}
  \end{equation}
  where we used the smoothness of $G_\alpha$ away from zero to apply the Leibniz integral rule and that $G_\alpha( \vx-\vy,0) = 0$ for all $\vy \in \Gamma$. Next, we use the identity $\partial G_\alpha / \partial t(\vx-\vy,t-\tau) = -\partial G_\alpha / \partial \tau(\vx-\vy,t-\tau)$ for the integrand.
  For a given $t \in (0,T)$ let $k$ be such that $t \in (t_{k-1},t_{k}]$, i.e.~the interval $\tau_{k}$ of the partition $I_{N_t}$ including its right endpoint. Then we can split up the temporal integral in \eqref{eq:time_derivative_slp} into integrals over $(0,t_{k-1})$ and $(t_{k-1},t)$. For the latter we get 
  \begin{align}
    &- \int_{t_{k-1}}^t \int_\Gamma \frac{\partial G_\alpha}{\partial\tau}( \vx-\vy,t-\tau) 
      \varphi(\vy,\tau) \vn(\vy) \,\dif\vs_\vy\,\dif\tau \label{eq:time_derivative_slp_ibp_time} \\
    & = \int_\Gamma G_\alpha(\vx-\vy,t-t_{k-1}) \varphi(\vy,t_{k-1}{+}) \vn(\vy)\, \dif\vs_y
    + \int_{t_{k-1}}^t \int_\Gamma G_\alpha(\vx-\vy,t-\tau) \frac{\partial \varphi}{\partial \tau}(\vy,\tau) \vn(\vy) 
        \,\dif\vs_\vy\,\dif\tau, \notag
  \end{align}
  using integration by parts, which is possible since $(\tau \mapsto \varphi(\vy,\tau)) \in C^1_\mathrm{pw}(I_{N_t})$ for almost all $\vy \in \Gamma$, and that $G_\alpha( \vx-\vy,0) = 0$ for all $\vy \in \Gamma$. Finally, we can rewrite the bilinear form
  \begin{align*}
    \left\langle \nvecm\cdot \dtracem \left( \frac{\partial}{\partial t} \widetilde{V} (\varphi \vn ) \right), \psi_m \right\rangle_{\Sigma_m} 
    = \sum_{k=1}^{N_t} \int_{t_{k-1}}^{t_k} &\int_{\Gamma_m} \psi_m(\vx,t) \vn_m(\vx)\cdot \frac{\partial}{\partial t} \widetilde{V} (\varphi \vn ) (\vx, t) \,\dif\vs_\vx\,\dif t
  \end{align*}
  and use \eqref{eq:time_derivative_slp} with $\partial G_\alpha/\partial t (\vx-\vy,t-\tau) = -\partial G_\alpha/\partial \tau(\vx - \vy,t-\tau)$, the splitting of the temporal integral and \eqref{eq:time_derivative_slp_ibp_time} to complete the proof of the lemma.
\end{proof}
\begin{proof}[Proof of Theorem~\ref{thm:second_bilinear_form_tensor_case}]
  Let $\{\Omega_m\}_m$, $\{\Lambda_m\}_m$ and $\{\omega_m\}_m$ be given as in Theorem~\ref{thm:approximating_domains}, $\Gamma_m := \partial \Omega_m$ and $\Sigma_m := \Gamma_m \times (0,T)$. For $\psi \in H^{1/2,1/4}(\Sigma) \cap L^\infty(\Sigma)$ let $\widetilde{\psi}$ be the unique solution of the initial boundary value problem~\eqref{eq:heat_equation}--\eqref{eq:dirichlet_boundary_values} with Dirichlet datum $\psi$. Note that $\widetilde{\psi}$ admits the representation 
  \begin{equation} \label{eq:caloric_extension_psi}
    \widetilde{\psi} = W ( (-1/2\, I + K)^{-1} \psi),  
  \end{equation}
  cf.~Theorem~\ref{thm:heat_solvability_l2}. In particular, $\widetilde{\psi} \in C^\infty(Q)$ and thus we can define $\psi_m := \widetilde{\psi}|_{\Sigma_m}$ in the classical sense. The idea of the proof is to show that for this choice of $\Sigma_m$ and $\psi_m$ the terms on the right-hand side of~\eqref{eq:second_bilinear_form_aux_boundary} converge to the respective terms of~\eqref{eq:second_bilinear_form_tensor_case} in the limit $m\rightarrow\infty$.

  We start with the first term on the right-hand side of \eqref{eq:second_bilinear_form_aux_boundary}. By transforming the integral over $\Gamma_m$ into an integral over $\Gamma$ we get
  \begin{equation}
    \begin{split}
      \sum_{k=1}^{N_t} &\int_{t_{k-1}}^{t_k} \int_{\Gamma_m} \int_0^{t_{k-1}} \int_\Gamma 
      \psi_m(\vx,t)\nvecm(\vx)\cdot\, \frac{\partial G_\alpha}{\partial\tau}(\vx-\vy,t-\tau) \varphi(\vy,\tau)\vn(\vy)
      \,\dif\vs_\vy\,\dif\tau\,\dif\vs_\vx\,\dif t \\
    &= \sum_{k=1}^{N_t} \int_{t_{k-1}}^{t_k} \int_{\Gamma} \int_0^{t_{k-1}} \int_\Gamma 
      \vh_m(\Lambda_m(\vx),\vy,t,\tau)\cdot \vf(\vy,\tau) \omega_m(\vx) \,\dif\vs_\vy\,\dif\tau\,\dif\vs_\vx\,\dif t,
      \label{eq:second_bilinearform_aux_boundary_transformed}
    \end{split}
  \end{equation}
  where we introduced the functions
  \begin{align}
    \vh_m(\vx,\vy,t,\tau) &:= \frac{\partial G_\alpha}{\partial\tau}(\vx-\vy,t-\tau) \label{eq:aux_def_h_m}
      \psi_m(\vx,t)\nvecm(\vx),\\ 
    \vf(\vy,\tau) &:= \varphi(\vy,\tau)\vn(\vy) \label{eq:aux_def_f}
  \end{align}
  and used the homeomorphism $\Lambda_m: \Gamma \rightarrow \Gamma_m$ and \emph{(iv)} of Theorem~\ref{thm:approximating_domains}. Let $\{Z_j\}_{j=1}^J$ be a finite family of coordinate cylinders covering~$\Gamma$ as in \emph{(ii)} of Theorem~\ref{thm:approximating_domains}, see also Definition~\ref{def:lipschitz_domain}. We split the outer spatial integrals in~\eqref{eq:second_bilinearform_aux_boundary_transformed} into segments 
  \begin{equation*}
    U_j:=\Gamma \cap Z_j
  \end{equation*}
  related to these coordinate cylinders. For this purpose, let $\{\Phi_j\}_{j=1}^J$ be a smooth partition of unity on $\Gamma$ subordinate to $\{Z_j\}_{j=1}^J$, i.e.~$\Phi_j \in C^\infty(\mathbb{R}^3)$, $\supp(\Phi_j) \subset Z_j$, $0 \leq \Phi_j \leq 1$ and $\sum_j \Phi_j(\vx)=1$ for all~$\vx \in \Gamma$. Then \eqref{eq:second_bilinearform_aux_boundary_transformed} is further equal to
  \begin{equation} \label{eq:second_bilinearform_aux_boundary_splitted}
    \sum_{k=1}^{N_t} \sum_{j=1}^J \int_{t_{k-1}}^{t_k} \int_{U_j} \int_0^{t_{k-1}} \int_\Gamma 
      \Phi_j(\vx) \vh_m(\Lambda_m(\vx),\vy,t,\tau)\cdot \vf(\vy,\tau) \omega_m(\vx) 
      \,\dif\vs_\vy\,\dif\tau\,\dif\vs_\vx\,\dif t.
  \end{equation}
  In particular, the convergence of the first term in \eqref{eq:second_bilinear_form_aux_boundary} to the first term in~\eqref{eq:second_bilinear_form_tensor_case} follows if we can show that
  \begin{equation} \label{eq:limit_first_term_piecewise}
    \begin{split}
      \lim_{m \rightarrow \infty}\Bigg( \int_{t_{k-1}}^{t_k} \int_{U_j} &\int_0^{t_{k-1}} \int_\Gamma 
        \Phi_j(\vx) \vh_m(\Lambda_m(\vx),\vy,t,\tau)\cdot \vf(\vy,\tau) \omega_m(\vx) 
        \,\dif\vs_\vy\,\dif\tau\,\dif\vs_\vx\,\dif t \Bigg) \\
      =\int_{t_{k-1}}^{t_k} \int_{U_j} &\int_0^{t_{k-1}} \int_\Gamma 
      \Phi_j(\vx) \vh(\vx,\vy,t,\tau)\cdot \vf(\vy,\tau) \,\dif\vs_\vy\,\dif\tau\,\dif\vs_\vx\,\dif t
    \end{split}
  \end{equation}
  for all $k \in \{1,\dots,N_t\}$ and $j \in \{1,\dots,J\}$, where 
  \begin{equation*}
    \vh(\vx,\vy,t,\tau) := \frac{\partial G_\alpha}{\partial\tau}(\vx-\vy,t-\tau) 
      \psi(\vx,t)\vn(\vx).
  \end{equation*}
  For this purpose, we split the inner integral of $\vy$ over $\Gamma$ into an integral over 
  \begin{equation*}
    V_j := 3Z_j \cap \Gamma  
  \end{equation*}
  and one over the remainder $\Gamma \backslash V_j$ and show the convergence of both parts using the classical dominated convergence theorem. 
  
  First, we consider
  \begin{equation} \label{eq:interior_integral_residual_y}
    \int_{t_{k-1}}^{t_k} \int_{U_j} \int_0^{t_{k-1}} \int_{\Gamma \backslash V_j} 
      \Phi_j(\vx) \vh_m(\Lambda_m(\vx),\vy,t,\tau)\cdot \vf(\vy,\tau) \omega_m(\vx) 
      \,\dif\vs_\vy\,\dif\tau\,\dif\vs_\vx\,\dif t
  \end{equation}
  and observe the pointwise convergence of $\omega_m(\vx) \vh_m(\Lambda_m(\vx),\vy,t,\tau)$ to $\vh(\vx,\vy,t,\tau)$ almost everywhere in the integration domain as $m \rightarrow \infty$. In fact, $\omega_m \rightarrow 1$ pointwise almost everywhere on $\Gamma$ by Theorem~\ref{thm:approximating_domains}~\emph{(iv)}. For the respective convergence of $\vh_m(\Lambda_m(\vx),\vy,t,\tau)$ in \eqref{eq:aux_def_h_m} to $\vh(\vx,\vy,t,\tau)$ we show the convergence of the individual terms. By~\emph{(iii)} of Theorem~\ref{thm:approximating_domains} we get that ${\nvecm(\Lambda_m(\vx)) \rightarrow \vn(\vx)}$ for almost all $\vx \in \Gamma$. Furthermore $\Lambda_m(\vx) \rightarrow \vx$ even uniformly on $\Gamma$, by~\emph{(i)}. Thus, ${\partial G_\alpha/\partial\tau(\Lambda_m(\vx)-\vy,t-\tau) \rightarrow \partial G_\alpha/\partial\tau(\vx-\vy,t-\tau)}$ for all $\vx, \vy \in \Gamma$ and $\tau < t$. Finally, we know that $\psi_m = \widetilde{\psi}|_{\Sigma_m}$ and that $\widetilde{\psi}$ given by~\eqref{eq:caloric_extension_psi} attains the Dirichlet boundary values $\psi$ on $\Sigma$ in the sense of non-tangential limits almost everywhere, see Theorem~\ref{thm:heat_solvability_l2}. Hence, $\psi_m(\Lambda_m(\vx),t) \rightarrow \psi(\vx,t)$ since $\Lambda_m(\vx)$ approaches $\vx$ non-tangentially by point~\emph{(i)} of Theorem~\ref{thm:approximating_domains}.

  It remains to find an integrable function that dominates the sequence of integrands. Obviously there holds $\Phi_j \leq 1$ on $U_j$ by construction and $|\vf(\vy,\tau)|\leq \|\varphi\|_{L^\infty(\Sigma)}$ almost everywhere on~$\Sigma$. By Theorem~\ref{thm:approximating_domains}~\emph{(iv)} we get in addition, that $\omega_m(\vx) \leq c^{-1}$ for a constant $c > 0$ independent of $m$ and all $\vx \in \Gamma$. Furthermore, $|\psi_m(\Lambda_m(\vx)) \nvecm(\Lambda_m(\vx))| \leq \|\psi\|_{L^\infty(\Sigma)}$ for almost all $\vx \in \Gamma$ by the extended maximum principle stated in Theorem~\ref{thm:extended_maximum_principle}. The only term left to bound is $\partial G_\alpha/\partial\tau(\Lambda_m(\vx)-\vy,t-\tau)$. Let $r$ and $h$ be the radius and height of the congruent cylinders ${Z_j}$ and let $\delta := \min(r,h)$. Due to~\eqref{eq:uniform_convergence_lambda_m} we can assume without loss of generality that $|\vx - \Lambda_m(\vx)| < \delta$ for all $\vx \in \Gamma$ and all $m$. Then we immediately get $\Lambda_m(\vx) \in 2Z_j$ for all $\vx \in U_j = \Gamma \cap Z_j$ and thus $|\Lambda_m(\vx) - \vy| > \delta$ for $\vx \in U_j$ and $\vy \in \Gamma \backslash V_j = \Gamma \backslash 3Z_j$. This allows us to bound
  \begin{equation*}
    \left|\frac{\partial G_\alpha}{\partial\tau}(\Lambda_m(\vx)-\vy,t-\tau)\right| \leq 
      \left\|\frac{\partial G_\alpha}{\partial t}\right\|_{L^\infty( (\overline{B_R(\vzero)}\backslash B_\delta(\vzero))
        \times [0,T])}
  \end{equation*}
  for all $\vx \in U_j$, $\vy \in \Gamma \backslash V_j$, and $0 < \tau < t < T$, where $B_R(\vzero)$ and $B_\delta(\vzero)$ are balls centered around the origin with radii $R$ and $\delta$, respectively, and $R$ is so large that $\overline{\Omega} \subset B_{R/2}(\vzero)$. The right-hand side in this estimate is bounded since $(\vr,t) \mapsto \partial G_\alpha / \partial t(\vr,t) \in C^\infty(\mathbb{R}^3\backslash \vzero \times [0,T])$. Altogether we have found the desired dominating function, as
  \begin{equation*}
    |\Phi_j(\vx) \vh_m(\Lambda_m(\vx),\vy,t,\tau)\cdot \vf(\vy,\tau) \omega_m(\vx)| \leq 
      \|\varphi\|_{L^\infty(\Sigma)} \|\psi\|_{L^\infty(\Sigma)} 
      \left\|\frac{\partial G_\alpha}{\partial t}\right\|_{L^\infty((\overline{B_R(\vzero)}\backslash B_\delta(\vzero))
        \times [0,T])}
  \end{equation*} 
  almost everywhere in the integration domain of \eqref{eq:interior_integral_residual_y}. Therefore, we have established the convergence of \eqref{eq:interior_integral_residual_y} to 
  \begin{equation} \label{eq:limit_interior_integral_residual_y}
    \int_{t_{k-1}}^{t_k} \int_{U_j} \int_0^{t_{k-1}} \int_{\Gamma \backslash V_j} 
      \Phi_j(\vx) \vh(\vx,\vy,t,\tau)\cdot \vf(\vy,\tau) 
      \,\dif\vs_\vy\,\dif\tau\,\dif\vs_\vx\,\dif t.
  \end{equation}

  To show \eqref{eq:limit_first_term_piecewise} we have to consider the remaining part
  \begin{equation} \label{eq:interior_integral_vj}
    \int_{t_{k-1}}^{t_k} \int_{U_j} \int_0^{t_{k-1}} \int_{V_j} 
      \Phi_j(\vx) \vh_m(\Lambda_m(\vx),\vy,t,\tau)\cdot \vf(\vy,\tau) \omega_m(\vx) 
      \,\dif\vs_\vy\,\dif\tau\,\dif\vs_\vx\,\dif t.
  \end{equation}
  We use the parametrization of the regions $\Lambda_m(U_j)$ and $V_j$ by the Lipschitz functions $\eta_j^{(m)}$ and $\eta_j$, respectively, established in Theorem~\ref{thm:approximating_domains}~\emph{(ii)}. This is possible, since $\Lambda_m(U_j) \subset \Gamma_m \cap 2Z_j$ as seen before. For the sake of simplicity we assume that the coordinates associated with the cylinder~$Z_j$ correspond to the original rectangular coordinates, i.e.~we neglect additional translations and rotations. Then \eqref{eq:interior_integral_vj} can be transformed into
  \begin{equation} \label{eq:interior_integral_vj_parametrized}
    \begin{split}
      \int_{t_{k-1}}^{t_k} \int_{B_{2r}(\vzero)} \int_0^{t_{k-1}} \int_{B_{3r}(\vzero)}
        \mathbbm{1}_{\mathcal{U}_j^{(m)}}(\hat{\vx})
        \Phi_j(\Lambda_m^{-1}&(\Gamma_m(\hat{\vx}))) 
        \vh_m( \Gamma_m(\hat{\vx}), \Gamma(\hat{\vy}), t, \tau )
        \\
        &\cdot \vf( \Gamma(\hat{\vy}), \tau ) g_m(\hat{\vx}) g(\hat{\vy})
        \, \dif\hat{\vx}\,\dif\tau\,\dif\hat{\vy}\,\dif t.
    \end{split}
  \end{equation} 
  where $\Gamma(\hat{\vy}):=(\hat{\vy}, \eta_j(\hat{\vy}))$, $\Gamma_m(\hat{\vx}):=(\hat{\vx}, \eta_j^{(m)}(\hat{\vx}))$, $r$ is still the radius of $Z_j$, $B_{3r}(\vzero)$ the parameter region of $V_j$, $\mathcal{U}_j^{(m)} \subset B_{2r}(\vzero)$ the parameter region of $\Lambda_m(U_j)$ and $g_m(\hat{\vx})$ and $g(\hat{\vy})$ denote the surface elements given by
  \begin{equation*}
    g_m(\hat{\vx}) = \sqrt{1 + |\nabla \eta_j^{(m)}(\hat{\vx})|^2}, \qquad 
    g(\hat{\vy}) = \sqrt{1 + |\nabla \eta_j(\hat{\vy})|^2}.
  \end{equation*}  
  We compute the limit of \eqref{eq:interior_integral_vj_parametrized} for $m \rightarrow \infty$ using the dominated convergence theorem again. 
  
  First we show that the integrand in \eqref{eq:interior_integral_vj_parametrized} converges to 
  \begin{equation} \label{eq:interior_integrand_vj_parametrized_limit}
    \mathbbm{1}_{B_r(\vzero)}(\hat{\vx}) \Phi_j(\Gamma(\hat{\vx}) )
    \vh(\Gamma(\hat{\vx}),\Gamma(\hat{\vy}), t, \tau)\cdot 
    \vf( \Gamma(\hat{\vy}), \tau ) g(\hat{\vx}) g(\hat{\vy})
  \end{equation}
  for almost all $\hat{\vx}$, $\hat{\vy}$, $t$ and $\tau$ in the integration domain as $m \rightarrow \infty$. Since ${\Gamma_m(\hat{\vx}) \neq \Lambda_m(\Gamma(\hat{\vx}))}$ in general, we cannot use the previous results about pointwise convergence to show this. Instead, we consider again all terms of the integrand in \eqref{eq:interior_integral_vj_parametrized} depending on $m$ separately to show pointwise convergence. Recall the definition of $h_m$ in \eqref{eq:aux_def_h_m} for this purpose. From Theorem~\ref{thm:approximating_domains}~\emph{(ii)} we know that~$\eta_j^{(m)}$ converges uniformly to $\eta_j$. As a result, $\Gamma_m(\hat{\vx}) \rightarrow \Gamma(\hat{\vx})$ for all $\hat{\vx} \in B_{2r}(\vzero)$ and thus $\partial G_\alpha / \partial \tau (\Gamma_m(\hat{\vx}),\Gamma(\hat{y}),t,\tau)$ converges to $\partial G_\alpha / \partial \tau (\Gamma(\hat{\vx}),\Gamma(\hat{y}),t,\tau)$ for all such $\hat{\vx}$, and~$\hat{\vy}$,~$t$ and~$\tau$ in the respective domains of integration. Since the convergence of~$\Gamma_m(\hat{\vx})$ to $\Gamma(\hat{\vx})$ is non-tangential it follows as before that $\psi_m(\Gamma_m(\hat{\vx}),t) \rightarrow \psi(\Gamma(\hat{\vx}),t)$ for almost all $\hat{\vx} \in B_{2r}(\vzero)$ and $t \in (t_{k-1},t_k)$. From Theorem~\ref{thm:approximating_domains}~\emph{(ii)} we know in addition that $\nabla \eta_j^{(m)}$ converges pointwise almost everywhere to $\nabla \eta_j$ in $\mathbb{R}^2$. This implies $g_m(\hat{\vx}) \rightarrow g(\hat{\vx})$ for almost all~$\hat{\vx}$ and furthermore the convergence of $\nvecm(\Gamma_m(\hat{\vx}))$ to $\vn(\Gamma(\hat{\vx}))$ for almost all $\hat{\vx} \in B_{2r}(\vzero)$, due to the representations 
  \begin{equation*}
    \nvecm(\Gamma_m(\hat{\vx})) = \frac{1}{g_m(\hat{\vx})} (\nabla \eta_j^{(m)}(\hat{\vx})^\top, -1)^\top, \qquad
    \vn(\Gamma(\hat{\vx})) = \frac{1}{g(\hat{\vy})} (\nabla \eta_j(\hat{\vx})^\top, -1 )^\top.  
  \end{equation*}
  
  Next we show the convergence of $\mathbbm{1}_{\mathcal{U}_j^{(m)}}$ to $\mathbbm{1}_{B_{r}(\vzero)}$ pointwise almost everywhere in $B_{2r}(\vzero)$. From~\eqref{eq:uniform_convergence_lambda_m} it follows that for all $\hat{\vx} \in B_{2r}(\vzero)$ and sufficiently small $\varepsilon > 0$ there exists an $m(\varepsilon)$ such that for all~$m>m(\varepsilon)$ 
  \begin{equation} \label{eq:preimage_of_cylinders}
    \Lambda_m^{-1}(\Gamma_m \cap Z_{\varepsilon/2}(\hat{\vx})) \subset (\Gamma \cap Z_{\varepsilon}(\hat{\vx})),
  \end{equation}
  where $Z_\rho(\hat{\vx}) := \{(\hat{\vxi},\hat{s}) \in \mathbb{R}^3 : |\hat{\vxi} - \hat{\vx}|<\rho, \hat{s} < h_j\}$ is the cylinder with center $(\hat{\vx},0)$, radius $\rho > 0$ and the same height $h_j$ as $Z_j$. If $\hat{\vx} \in B_r(\vzero)$ and $\varepsilon$ is so small that $(\Gamma \cap Z_{\varepsilon}(\hat{\vx})) \subset U_j$, \eqref{eq:preimage_of_cylinders} implies that $\Gamma_m(\hat{\vx}) \in \Lambda_m(U_j)$ for all $m > m(\varepsilon)$. This means that $\hat{\vx} \in \mathcal{U}_j^{(m)}$ and thus $\mathbbm{1}_{\mathcal{U}_j^{(m)}}(\hat{\vx}) = 1 = \mathbbm{1}_{B_{r}(\vzero)}(\hat{\vx})$ for all~$m > m(\varepsilon)$. Likewise, if $\hat{\vx} \in B_{2r}(\vzero) \backslash \overline{B_r(\vzero)}$ and $\varepsilon$ is so small that $(\Gamma \cap Z_{\varepsilon}(\hat{\vx})) \cap U_j = \emptyset$, we get ${\mathbbm{1}_{\mathcal{U}_j^{(m)}}(\hat{\vx}) = 0 = \mathbbm{1}_{B_{r}(\vzero)}(\hat{\vx})}
  $ for all $m > m(\varepsilon)$. Together this proves $\mathbbm{1}_{\mathcal{U}_j^{(m)}} \rightarrow \mathbbm{1}_{B_{r}(\vzero)}$ almost everywhere in $B_{2r}(\vzero)$.

  Finally we have to show that $\Phi_j(\Lambda_m^{-1}(\Gamma_m(\hat{\vx})))$ converges to $\Phi_j(\Gamma(\hat{\vx}) )$ in $B_{2r}(\vzero)$. Since $\Phi_j$ is continuous it suffices to show the convergence of its arguments. This follows again from~\eqref{eq:preimage_of_cylinders}. Indeed, for each $\varepsilon > 0$ we know by \eqref{eq:preimage_of_cylinders} that for all $m>m(\varepsilon)$ there exists a $\hat{\vy}_m$ such that $\Lambda_m^{-1}(\Gamma_m(\hat{\vx})) = \Gamma(\hat{\vy}_m)$
  and $|\hat{\vy}_m - \hat{\vx}| < \varepsilon$. As a consequence there holds
  \begin{equation*}
    |\Lambda_m^{-1}(\Gamma_m(\hat{\vx})) - \Gamma(\hat{\vx})| = |\Gamma(\hat{\vy}_m) - \Gamma(\hat{\vx})| 
    = |(\hat{\vy}_m,\eta_j(\hat{\vy}_m)) - (\hat{\vx},\eta_j(\hat{\vx}))| \leq \varepsilon \sqrt{1+L_j^2},
  \end{equation*}
  where $L_j$ is the Lipschitz constant of $\eta_j$. This yields the desired convergence. In particular, we have shown that the integrand in \eqref{eq:interior_integral_vj_parametrized} converges pointwise almost everywhere to \eqref{eq:interior_integrand_vj_parametrized_limit}.

  For the application of the dominated convergence theorem we bound the integrands in \eqref{eq:interior_integral_vj_parametrized} uniformly by using the estimates
  \begin{align*}
    |\mathbbm{1}_{\mathcal{U}_j^{(m)}}(\hat{\vx}) \Phi_j(\Lambda_m^{-1}(\Gamma_m(\hat{\vx})))| &\leq 1, \\
    |\varphi(\Gamma(\hat{\vy}),\tau)\vn(\Gamma(\hat{\vy}))| & \leq \|\varphi\|_{L^\infty(\Sigma)},\\
    |\psi_m(\Gamma_m(\hat{\vx}),t)\nvecm(\Gamma_m(\hat{\vx}))| & \leq \|\psi\|_{L^\infty(\Sigma)}, \\
    |g(\hat{\vy})| &\leq \sqrt{1 + \|\nabla \eta_j\|_{L^\infty(\mathbb{R}^2)}}, \\
    |g_m(\hat{\vx})|&\leq \sqrt{1 + \|\nabla \eta_j^{(m)}\|_{L^\infty(\mathbb{R}^2)}} \leq 
      \sqrt{1 + \|\nabla \eta_j\|_{L^\infty(\mathbb{R}^2)}}
  \end{align*}
  for almost all $\hat{\vx}$, $\hat{\vy}$, $t$ and $\tau$ in the integration domain. The first estimate is clear by definition. The second one holds true due to the assumption that $\varphi \in L^\infty(\Sigma)$. The third estimate is again a consequence of the parabolic maximum principle in Theorem~\ref{thm:extended_maximum_principle}. The last two estimates follow from the definition of the surface elements and the fact that $\|\nabla \eta_j^{(m)}\|_{L^\infty(\mathbb{R}^2)} \leq \|\nabla \eta_j\|_{L^\infty(\mathbb{R}^2)}$, see Theorem~\ref{thm:approximating_domains}~\emph{(ii)}. Therefore, the product of all these functions is bounded by a constant $C$ independent of $m$. The only term left to consider is the derivative of the heat kernel
  \begin{equation*}
    \frac{\partial G_\alpha}{\partial\tau}(\vx-\vy,t-\tau) =
    \left[ \frac{6\alpha(t-\tau) - |\vx -\vy|^2 }{(4 \alpha)^{5/2}\pi^{3/2}(t-\tau)^{7/2}} \right]
    \exp\bigg( -\frac{|\vx-\vy|^2}{4\alpha(t-\tau)} \bigg) \qquad \text{for } t > \tau.
  \end{equation*}
  By considering only the first part of the numerator we can estimate \cite[cf.~Chapter~13 \S3]{Pog1966}
  \begin{equation} \label{eq:est_part_time_derivative_heat_kernel}
    \begin{split}
      \frac{6\,\alpha \,s}{(4 \alpha)^{5/2}\pi^{3/2}s^{7/2}} \exp\bigg( -\frac{|\vr|^2}{4\alpha\,s} \bigg)
      = \left(\frac{|\vr|^2}{4\alpha s}\right)^{3/4} \exp \left( -\frac{|\vr|^2}{4\alpha s} \right)
        \frac{3}{2\, \pi^{3/2}(4\alpha)^{3/4} s^{7/4} |\vr|^{3/2}} &\\
       \leq \left(\frac{3}{4}\right)^{3/4} \exp(-3/4) \frac{3}{2\, \pi^{3/2}(4\alpha)^{3/4} s^{7/4} |\vr|^{3/2}}
      = c(\alpha) \frac{1}{s^{7/4}} \frac{1}{|\vr|^{3/2}} &
    \end{split}
  \end{equation}
  for $s=t-\tau>0$, where we used that $q^m \exp(-q) \leq m^m \exp(-m)$ for all $m, q \geq 0$. The expression corresponding to the second part of the numerator can be handled similarly and we end up with the estimate 
  \begin{equation} \label{eq:estimate_temporal_derivative_G_alpha}
    \left|\frac{\partial G_\alpha}{\partial \tau}(\Gamma_m(\hat{\vx}),\Gamma(\hat{\vy}),t,\tau)\right| 
    \leq \frac{\tilde{c}(\alpha)}{(t-\tau)^{7/4} |\Gamma_m(\hat{\vx})-\Gamma(\hat{\vy})|^{3/2}} 
    \leq \frac{\tilde{c}(\alpha)}{(t-\tau)^{7/4}|\hat{\vx}-\hat{\vy}|^{3/2}} 
  \end{equation}
  almost everywhere in the considered integration domain $B_{2r}(\vzero) \times (t_{k-1},t_k) \times B_{3r}(\vzero) \times (0,t_{k-1})$ in~\eqref{eq:interior_integral_vj_parametrized}. Since the function on the right-hand side of~\eqref{eq:estimate_temporal_derivative_G_alpha} is integrable on this domain we have found a function dominating the sequence of integrands. 
  
  By the dominated convergence theorem we get that \eqref{eq:interior_integral_vj_parametrized} converges to 
  \begin{align*}
    \int_{t_{k-1}}^{t_k} &\int_{B_{r}(\vzero)} \int_0^{t_{k-1}} \int_{B_{3r}(\vzero)}
      \Phi_j(\Gamma(\hat{\vx}) )
      \vh(\Gamma(\hat{\vx}),\Gamma(\hat{\vy}), t, \tau)\cdot 
      \vf( \Gamma(\hat{\vy}), \tau ) g(\hat{\vx}) g(\hat{\vy}) \, \dif\hat{\vx}\,\dif\tau\,\dif\hat{\vy}\,\dif t \\
    &= \int_{t_{k-1}}^{t_k} \int_{U_j} \int_0^{t_{k-1}} \int_{V_j} 
    \Phi_j(\vx) \vh(\vx,\vy,t,\tau)\cdot \vf(\vy,\tau) 
    \,\dif\vs_\vy\,\dif\tau\,\dif\vs_\vx\,\dif t.
  \end{align*}
  Together with the convergence of \eqref{eq:interior_integral_residual_y} to \eqref{eq:limit_interior_integral_residual_y} we conclude \eqref{eq:limit_first_term_piecewise}. 
  
  Recall that we wanted to show that all terms on the right-hand side of \eqref{eq:second_bilinear_form_aux_boundary} converge to the respective terms of \eqref{eq:second_bilinear_form_tensor_case}. Equation~\eqref{eq:limit_first_term_piecewise} implies the convergence of the first term. In particular, the integrand of the first term on the right-hand side of \eqref{eq:second_bilinear_form_tensor_case} is Lebesgue integrable in the corresponding integration domain, which is another consequence of the dominated convergence theorem. The remaining two terms can be handled analogously. For both terms one can transform the integrals of $\vx$ over~$\Gamma_m$ into integrals over $\Gamma$ as in \eqref{eq:second_bilinearform_aux_boundary_transformed} and split the integrals up as in \eqref{eq:second_bilinearform_aux_boundary_splitted}. The individual parts can then be handled as before by splitting the inner integral of $\vy$ over $\Gamma$ into integrals over~$V_j$ and $\Gamma \backslash V_j$. For both one applies the dominated convergence theorem where one uses the estimate
  \begin{equation*}
    |G_\alpha(\vr,t)|\leq c(\alpha) \frac{1}{t^{3/4}} \frac{1}{|\vr|^{3/2}},
  \end{equation*}
  which can be shown as in \eqref{eq:est_part_time_derivative_heat_kernel}, to bound the heat kernel and in addition the estimates
  \begin{align*}
    |\varphi(\vy,t_{k-1}{+})| &\leq \|\varphi(\cdot,t_{k-1}{+})\|_{L^\infty(\Gamma)}, \\ 
    \left|\frac{\partial\varphi}{\partial\tau}(\vy,\tau)\right| &\leq \left\|\frac{\partial\varphi}{\partial\tau}\right\|_{L^\infty(\Gamma \times (t_{k-1},t_k))}
  \end{align*}
  for $\vy \in \Gamma$ and $\tau \in (t_{k-1},t_k)$. Note that these estimates are reasonable due to the assumption that $\varphi \in (L^\infty(\Gamma) \cap H^{1/2}(\Gamma)) \otimes C^1_\mathrm{pw}(I_{N_t})$ with $C^1_\mathrm{pw}(I_{N_t})$ defined in \eqref{eq:def_c1_pw}.
\end{proof}

\section{Conclusion} \label{sec:conclusion}
The integration by parts formula for the bilinear form of the hypersingular boundary integral operator is a key result when it comes to its evaluation in Galerkin methods. In this paper we have provided a general version of this formula for the 3+1D transient heat equation in~\eqref{eq:general_ibp_formula} together with a rigorous proof, which was missing in the literature to the best of our knowledge. However, the general formula is not sufficient for the evaluation of the bilinear form for non-smooth functions since it includes the bilinear form $b(\cdot, \cdot)$ in \eqref{eq:def_b}, which is defined only as a continuous extension. We have shown that the usual interpretation of this bilinear form in the literature is problematic and provided a suitable alternative for certain types of functions including the typical tensor product discretization spaces. As a side result we have provided a proof of a generalization of the classical parabolic maximum principle in Theorem~\ref{thm:extended_maximum_principle}.

An alternative strategy for the discretization of the boundary integral operators for the heat equation is to discretize the space-time boundary $\Sigma$ with tetrahedral instead of tensor product meshes and to consider related discrete function spaces. The representation for the bilinear form~$b(\cdot,\cdot)$ provided in Section~\ref{sec:bf_time_derivative} is not suitable for such function spaces. An alternative representation might, however, be derived in a similar way as in that section.

\section*{Acknowledgements}
The authors acknowledge the support provided by the Austrian Science Fund (FWF) under the project I 4033-N32 in a joint project with the Czech Science Foundation (project 17-22615S) and the suggestion of the simpler version of the proof of Proposition~\ref{prop:kernel_not_integrable} from an anonymous reviewer.

\bibliographystyle{abbrvurl}
\bibliography{references}

\appendix
\section{Appendix}
\subsection{Lipschitz domains and their smooth approximation}
In this paper we use the following definition of Lipschitz domains similar to \cite[Definition~2.1]{Bro1989}.
\begin{definition}[Lipschitz domain] \label{def:lipschitz_domain}
  A set $Z_j \subset \mathbb{R}^3$ is called an \emph{open coordinate cylinder} with radius $r_j > 0$ and height $h_j > 0$ if there exist a rectangular coordinate system of $\mathbb{R}^3$ obtained from the standard Cartesian coordinate system by rotation and translation with corresponding coordinates $\vx^{(j)} \in \mathbb{R}^2$ and $s^{(j)} \in \mathbb{R}$ such that
  \begin{equation*}
    Z_j = \{ (\vx^{(j)},s^{(j)}) \in \mathbb{R}^3 : |\vx^{(j)}|<r_j, \quad |s^{(j)}| < h_j \}.
  \end{equation*}
  A bounded, connected, open subset $\Omega \subset \mathbb{R}^3$ is called a \emph{Lipschitz domain}, if there exists a finite family $\{Z_j\}_j$ of open coordinate cylinders covering $\Gamma := \partial \Omega$ and for each $j$ there exists a Lipschitz--continuous function $\eta_j : \mathbb{R}^2 \rightarrow \mathbb{R}$, i.e.~$|\eta_j(\vx)-\eta_j(\vy)| \leq L_j\, |\vx - \vy|$ for some $L_j \in \mathbb{R}$, such that $|\eta_j(\vx)| < h_j$ and
  \begin{equation} \label{eq:lipschitz_hypergraph}
    \begin{split}
      \Omega \cap Z_j &= \{ (\vx^{(j)},s^{(j)}) \in Z_j : s^{(j)} > \eta_j(\vx^{(j)}) \}, \\
      \Gamma \cap Z_j &= \{ (\vx^{(j)}, \eta_j(\vx^{(j)})) : \vx^{(j)} \in \mathbb{R}^2 \} \cap Z_j,
    \end{split}
  \end{equation}
  where $(\vx^{(j)},s^{(j)})$ denote the coordinates associated with $Z_j$ and $h_j$ its height.
\end{definition}
\begin{remark} \label{rem:dilatable_congruent_covering}
  For each Lipschitz domain $\Omega$ we can find a family of coordinate cylinders $\{Z_j\}_j$ covering $\Gamma$ as in Definition~\ref{def:lipschitz_domain} such that for all $j$ the dilated cylinder
  \begin{equation*}
    3 Z_j := \{ (\vx^{(j)},s^{(j)}) \in \mathbb{R}^3 : |\vx^{(j)}|<3r_j, |s^{(j)}|<3h_j  \}
  \end{equation*}
  satisfies \eqref{eq:lipschitz_hypergraph} too, see \cite[Definition~2.1]{Bro1989} and \cite[Section~0.2]{Ver1984}. Furthermore we can choose the cylinders to be congruent, i.e.~to have the same radii and heights. 
\end{remark}
The following result provides us with a sequence $\{\Omega_m\}_m$ of smooth domains approximating a given Lipschitz domain $\Omega$ from the inside. 
\begin{theorem}[\!\!{\cite[Lemma~2.2]{Bro1989}, \cite[Theorem~A.1]{Ver1982}, \cite[Theorem~1.12]{Ver1984}}] \label{thm:approximating_domains}
  Let $\Omega$ be a Lipschitz domain with boundary $\Gamma$. Then there exist sequences of $C^\infty$ domains $\{\Omega_m\}_m$, homeomorphisms $\{\Lambda_m\}_m$ and functions~$\{\omega_m\}_m$ that satisfy:
  \begin{enumerate}[(i)]
    \item $\overline{\Omega}_m \subset \Omega$ and the homeomorphisms $\Lambda_m: \Gamma \rightarrow \Gamma_m := \partial \Omega_m$ satisfy 
    \begin{equation} \label{eq:uniform_convergence_lambda_m}
      \lim_{m\rightarrow \infty}\left(\sup\{ |\vx - \Lambda_m(\vx)| : \vx \in \Gamma \} \right) = 0.
    \end{equation}
    In addition, $\Lambda_m(\vx)$ approaches $\vx$ non-tangentially.
    \item There exists a finite family of coordinate cylinders $\{Z_j\}_j$ covering $\Gamma$ as in Remark~\ref{rem:dilatable_congruent_covering} and associated Lipschitz functions $\{\eta_j\}_j$ such that for each $j$ and $m$ there exists a function $\eta_j^{(m)} \in C^\infty(\mathbb{R}^2)$ that represents $\Gamma_m$ in $3Z_j$, i.e.
    \begin{equation*}
      \Gamma_m \cap 3 Z_j = \{ (\vx^{(j)}, \eta_j^{(m)}(\vx^{(j)})) : \vx^{(j)} \in \mathbb{R}^2 \} \cap 3 Z_j.
    \end{equation*} 
    Furthermore, $\eta_j^{(m)} \rightarrow \eta_j$ uniformly and $\nabla \eta_j^{(m)} \rightarrow \nabla \eta_j$ pointwise almost everywhere as $m$ tends to infinity, and $\|\eta_j^{(m)}\|_{L^\infty(\mathbb{R}^2)} \leq \|\eta_j\|_{L^\infty(\mathbb{R}^2)}$ for all $m$.
    \item The normal vectors $\nvecm$ on $\Gamma_m$ convergence pointwise almost everywhere to the normal vector $\vn$ on $\Gamma$, in the sense that $\nvecm(\Lambda_m(\vx)) \rightarrow \vn(\vx)$ for almost all $\vx \in \Gamma$ as $m$ tends to infinity. 
    \item The functions $\omega_m: \Gamma \rightarrow \mathbb{R}_{>0}$ are such that 
    \begin{equation*}
      \int_{E} \omega_m(\vx) \dif\vs_\vx = \int_{\Lambda_m(E)} \dif\vs_\vy
    \end{equation*}
    for all measurable sets $E \subset \Gamma$, where $\dif\vs_\vx$ and $\dif\vs_\vy$ denote the surface measures on $\Gamma$ and~$\Gamma_m$, respectively. Furthermore, there exists a constant $c>0$ such that ${c \leq \omega_m \leq c^{-1}}$ for all $m$ and $\omega_m \rightarrow 1$ pointwise almost everywhere as $m$ tends to infinity.
  \end{enumerate}
\end{theorem}
\begin{remark}
  In item \emph{(i)} of Theorem~\ref{thm:approximating_domains} we stated that $\Lambda_m(\vx)$ approaches $\vx$ non-tangentially. Roughly speaking this means that the points $\Lambda_m(\vx)$ all lie in a cone centered at $\vx$. A rigorous definition is given in \cite{Bro1989,Ver1984}. Some additional properties of the approximating domains $\{\Omega_m\}_m$ are given in the referenced works, and a proof of the theorem can be found in \cite[Theorem~A.1]{Ver1982}.
\end{remark}

\ifshort
  \subsection{Postponed proofs} 
\fi
\label{sec:appendix_proofs}
Here we collect the postponed proofs in the order in which they appeared in the paper.
\begin{proof}[Sketch of the proof of Proposition \ref{prop:bilinear_form_d}]
  Let $u$ and $v$ be in $C^\infty_c(\overline{\Omega}\times(0,T])$ and $C^\infty_c(\overline{\Omega}\times[0,T))$, respectively. We can extend $u$ to a function $\widetilde{u} \in H^1(\mathbb{R};L^2(\Omega))$ by setting $\widetilde{u}(t,\cdot)=0$ for all $t<0$ and then setting $\widetilde{u}(\cdot,t)=\widetilde{u}(\cdot,2T-t)$ for all $t>T$. For this extension one can show that
  \begin{equation*}
    \|\widetilde{u}\|_{H^{1/2}(\mathbb{R};L^2(\Omega))} \leq c(T) \|u\|_{H^{1,1/2}_{;0,}(Q)},
  \end{equation*}
  where the constant $c(T)$ does only depend on $T$. Similarly, we can extend $v$ to a function $\widetilde{v} \in H^{1}(\mathbb{R};L^2(\Omega))$ such that 
  \begin{equation*}
    \|\widetilde{v}\|_{H^{1/2}(\mathbb{R};L^2(\Omega))} \leq c(T) \|v\|_{H^{1,1/2}_{;,0}(Q)}.
  \end{equation*}
  Since $u$ and $v$ are compactly supported in $\overline{\Omega}\times(0,T]$ and $\overline{\Omega}\times[0,T)$, respectively, there holds 
  \begin{equation} \label{eq:boundedness_d}
    \begin{split}
      d(u,v) &= \int_0^T \int_\Omega \frac{\partial u}{\partial t} v \,\dif\vx\,\dif t
      = \int_\mathbb{R} \int_\Omega \frac{\partial \widetilde{u}}{\partial t} \widetilde{v} \, \dif\vx\,\dif t
        \leq c \|\widetilde{u}\|_{H^{1/2}(\mathbb{R};L^2(\Omega))} \|\widetilde{v}\|_{H^{1/2}(\mathbb{R};L^2(\Omega))}\\
        &\leq c_d(T) \|u\|_{H^{1,1/2}_{;0,}(Q)} \|v\|_{H^{1,1/2}_{;,0}(Q)}.
    \end{split}
  \end{equation}
  The first estimate here can be shown by switching to the Fourier domain in time using Plan\-che\-rel's theorem, where the estimate follows easily when considering the equivalent norms in~$H^{1/2}(\mathbb{R};L^2(\Omega))$ defined via Fourier transforms. The second estimate is a consequence of the two estimates above. By a density result similar to the one in the sketch of the proof of Proposition~\ref{prop:density_anisotropic_Q} the assertion follows.
\end{proof}

For the next proof we need to introduce another anisotropic Sobolev space, namely
\begin{equation*}
  H^{1,1/2}_{0;}(Q) := L^2(0,T;H^1_0(\Omega)) \cap H^{1/2}(0,T;L^2(\Omega)),
\end{equation*}
which can be understood as the subspace of $H^{1,1/2}(Q)$ whose functions vanish on~$\Sigma$. In a similar way as in Proposition~\ref{prop:density_anisotropic_Q} it can be shown that $C^\infty_c(Q)$ is dense in $H^{1,1/2}_{0;}(Q)$.
\begin{proof}[Proof of Proposition~\ref{prop:surface_curl}] 
  We start by showing that the definition in \eqref{eq:def_scurl} is independent of the extension $\rinvdtrace \vpsi$ of the test function $\vpsi \in \mathbf{H}^{1/2,1/4}(\Sigma)$. For this purpose, let $\widetilde{\vpsi}_1$ and $\widetilde{\vpsi}_2$ in $\mathbf{H}^{1,1/2}(Q)$ denote two extensions of $\vpsi$ to~$Q$. Then, the difference $\widetilde{\vpsi}_1 - \widetilde{\vpsi}_2$ is in $\mathbf{H}^{1,1/2}_{0;}(Q)$ and therefore
  \begin{equation} \label{eq:scurl_independence_extension_1}
    \langle \nabla u, \curl \widetilde{\vpsi}_1 - \curl \widetilde{\vpsi}_2 \rangle_{\mathbf{L}^2(Q)}  = 0
  \end{equation}
  for all $u \in H^{1,1/2}(Q)$. Indeed, integration by parts yields
  \begin{equation*}
    \langle \nabla u, \curl \vw \rangle_{\mathbf{L}^2(Q)} = \langle u, \vn\cdot \curl \vw \rangle_\Sigma - \langle u, \diver (\curl \vw) \rangle_{L^2(Q)} = 0
  \end{equation*}
  for $\vw \in \mathbf{C}^\infty_c(Q)$ and thus \eqref{eq:scurl_independence_extension_1} follows from the density of $\mathbf{C}^\infty_c(Q)$ in $\mathbf{H}^{1,1/2}_{0;}(Q)$. This proves that \eqref{eq:def_scurl} is independent of the extension $\rinvdtrace \vpsi$ of $\vpsi$.

  Similarly, we conclude that \eqref{eq:def_scurl} is independent of the extension $\rinvdtrace \varphi$ of $\varphi$ by using that
  \begin{equation*}
    \langle \nabla u, \curl \vw \rangle_{\mathbf{L}^2(Q)} = \langle \nabla u \times \vn, \vw \rangle_\Sigma + \langle \curl (\nabla u), \vw \rangle_{\mathbf{L}^2(Q)} = 0
  \end{equation*}
  for all $u \in C^\infty_c(Q)$ and $\vw \in \mathbf{H}^{1,1/2}(Q)$.

  To see that $\scurl \varphi \in \mathbf{H}^{-1/2,-1/4}(\Sigma)$ and that $\scurl$ is continuous as a mapping from $H^{1/2,1/4}(\Sigma)$ to $\mathbf{H}^{-1/2,-1/4}(\Sigma)$ we estimate 
  \begin{align*}
    |\langle \scurl \varphi, \vpsi \rangle_\Sigma|
      &= |\langle \nabla \rinvdtrace \varphi, \curl( \rinvdtrace \vpsi) \rangle_{\mathbf{L}^2(Q)}| 
      \leq \| \nabla \rinvdtrace \varphi \|_{\mathbf{L}^2(Q)} \| \curl( \rinvdtrace \vpsi) \|_{\mathbf{L}^2(Q)} \\
      &\leq c\, \| \rinvdtrace \varphi \|_{H^{1,1/2}(Q)} \| \rinvdtrace \vpsi \|_{\mathbf{H}^{1,1/2}(Q)}
       \leq c \, c_\mathrm{IT}^2 \| \varphi \|_{H^{1/2,1/4}(\Sigma)} \| \vpsi \|_{\mathbf{H}^{1/2,1/4}(\Sigma)},
  \end{align*}
  where $c_\mathrm{IT}$ denotes the boundedness constant of the extension operators $\rinvdtrace$.

  The only thing left to show is \eqref{eq:scurl_smooth_functions} for $\widetilde{\varphi} \in C^2(\overline{Q})$ and $\varphi = \widetilde{\varphi}|_{\Sigma}$. Since the definition of~$\scurl$ in~\eqref{eq:def_scurl} is independent of the extension of $\varphi$ we can use the particular extension $\widetilde{\varphi}$ to get
  \begin{align*}
    \langle \scurl \varphi, \vpsi \rangle_\Sigma 
    &= \langle \nabla \widetilde{\varphi}, \curl( \rinvdtrace\vpsi) \rangle_{\mathbf{L}^2(Q)}
    = \langle \nabla \widetilde{\varphi} \times \vn, \vpsi \rangle_\Sigma 
      + \langle \curl (\nabla \widetilde{\varphi}), \rinvdtrace \vpsi \rangle_{\mathbf{L}^2(Q)} \\
    &= \langle \nabla \widetilde{\varphi} \times \vn, \vpsi \rangle_\Sigma
  \end{align*}
  for all $\vpsi \in \mathbf{H}^{1/2,1/4}(\Sigma)$, where we used integration by parts in the second step. In particular, $\scurl \varphi = \nabla \widetilde{\varphi} \times \vn$ in $\mathbf{H}^{-1/2,-1/4}(\Sigma)$.
\end{proof}

\begin{proof}[Proof of Lemma \ref{lem:curl_and_scurl}]
  For all $\vw \in \mathbf{C}^\infty_c(\mathbb{R}^3\times \mathbb{R})$ there holds
  \begin{equation*}
    \curl ( \dtrace )'(\varphi \vn )[\vw] = ( \dtrace )'(\varphi \vn )[\curl \vw]
      = \int_0^T \int_\Gamma \varphi(\vx,t) \vn(\vx)\cdot \curl \vw(\vx,t) \,\dif\vs_\vx\,\dif t.  
  \end{equation*}
  Let us first assume that $\varphi \in \dtrace(C^\infty_c(\overline{\Omega} \times (0,T)))$, i.e.~there exists $\widetilde{\varphi} \in C^\infty_c(\overline{\Omega} \times (0,T))$ such that $\varphi=\widetilde{\varphi}|_{\Sigma}$. Then we can rewrite
  \begin{equation*}
    \varphi(\vx,t) \curl \vw(\vx,t) 
      = \curl( \widetilde{\varphi} \vw)(\vx,t) 
      - \nabla \widetilde{\varphi}(\vx,t) \times \vw(\vx,t)
  \end{equation*}
  for all $\vx \in \Gamma$ and $t \in (0,T)$. By inserting this into the previous equation we get
  \begin{align*}
    \int_0^T \int_\Gamma \vn(\vx)\cdot \curl( \widetilde{\varphi} \vw)(\vx,t) \,\dif\vs_\vx\,\dif t
    - \int_0^T \int_\Gamma \vn(\vx)\cdot (\nabla \widetilde{\varphi}(\vx,t) \times \vw(\vx,t) ) 
      \,\dif\vs_\vx\,\dif t  
  \end{align*}
  The first integral vanishes, which follows by applying the divergence theorem and using that $\diver  \curl( \widetilde{\varphi} \vw) = 0$. The integrand of the second integral is
  \begin{equation*}
    \vn(\vx)\cdot (\nabla \widetilde{\varphi}(\vx, t) \times \vw(\vx, t) )
      = - \vw(\vx,t) \cdot (\nabla \widetilde{\varphi}(\vx,t) \times \vn(\vx))
      = - \vw(\vx,t) \cdot (\scurl \varphi)
  \end{equation*}
  where we used \eqref{eq:scurl_smooth_functions} in the last step. Hence,
  \begin{equation*}
    \curl ( \dtrace )'(\varphi \vn )[\vw] 
      = \int_0^T \int_\Gamma \vw(\vx,t) \cdot (\scurl \varphi)(\vx,t) \,\dif\vs_\vx\,\dif t  
      = (\dtrace)'(\scurl \varphi)[\vw]
  \end{equation*}
  and thus \eqref{eq:curl_and_scurl} holds for all $\varphi\in \dtrace(C^\infty_c(\overline{\Omega} \times (0,T)))$.

  Let us now consider a general $\varphi \in H^{1/2,1/4}(\Sigma)$. The space $\dtrace(C^\infty_c(\overline{\Omega} \times (0,T)))$ is dense in $H^{1/2,1/4}(\Sigma)$, which follows from Proposition~\ref{prop:density_anisotropic_Q}. Therefore, we can find a sequence $\{\varphi_k\}_k$ in $\dtrace(C^\infty_c(\overline{\Omega} \times (0,T)))$ such that $\varphi_k \rightarrow \varphi$ in $H^{1/2,1/4}(\Sigma)$ as $k \rightarrow \infty$. It follows that
  \begin{equation*}
    \curl ( \dtrace )'(\varphi_k \vn ) \rightarrow \curl ( \dtrace )'(\varphi \vn )
  \end{equation*} 
  in the distributional sense as $k \rightarrow \infty$. In fact, for all $\vw \in \mathbf{C}^\infty_c(\mathbb{R}^3\times\mathbb{R})$ there holds
  \begin{align*}
    \lim_{k\rightarrow\infty} \curl( \dtrace )'(\varphi_k \vn )[\vw] 
    &= \lim_{k\rightarrow\infty} \int_0^T \int_\Gamma \varphi_k(\vx,t) \vn(\vx)\cdot \curl \vw(\vx,t) \,\dif\vs_\vx\,\dif t \\
    &= \lim_{k\rightarrow\infty} \langle \varphi_k, \vn \cdot \curl \vw \rangle_{L^2(\Sigma)} 
    = \langle \varphi, \vn \cdot \curl \vw \rangle_{L^2(\Sigma)} \\
    &= \curl( \dtrace )'(\varphi \vn )[\vw]. 
  \end{align*}
  At the same time 
  \begin{align*}
    \lim_{k\rightarrow\infty} \curl( \dtrace )'(\varphi_k \vn )[\vw] 
      &= \lim_{k\rightarrow\infty} ( \dtrace )'(\scurl \varphi_k )[\vw]
      = \lim_{k\rightarrow\infty} \langle \scurl \varphi_k, \dtrace \vw \rangle_\Sigma \\
      &= \langle \scurl \varphi, \dtrace \vw \rangle_\Sigma
      = ( \dtrace )'(\scurl \varphi )[\vw],
  \end{align*}
  for all $\vw \in \mathbf{C}^\infty_c(\mathbb{R}^3\times\mathbb{R})$ due to the continuity of $\scurl$ in $H^{1/2,1/4}(\Sigma)$. Hence
  \begin{equation*}
    \curl( \dtrace )'(\varphi \vn ) = (\dtrace)'(\scurl \varphi) 
  \end{equation*}
  holds for general $\varphi \in H^{1/2,1/4}(\Sigma)$.
\end{proof}

\end{document}